\documentclass[12pt]{article}
\usepackage[utf8]{inputenc}
\usepackage[T1]{fontenc}
\usepackage{lmodern}
\usepackage{graphicx}
\newcommand{\vs}{\vskip10pt}

\usepackage{amsmath}
\usepackage{amsthm}
\usepackage{amssymb}
\usepackage[export]{adjustbox}
\usepackage[margin=0.75in]{geometry}
\usepackage{setspace}
\usepackage{enumitem}
\usepackage{float}
\usepackage{tabularx,ragged2e}
\newtheorem{thm}{Theorem}[section]

\newtheorem{cor}[thm]{Corollary}

\newtheorem{rem}[thm]{Remark}
\newtheorem{lem}[thm]{Lemma}
\newtheorem*{lem*}{Lemma}
\newtheorem*{thm*}{Theorem}
\newtheorem*{cor*}{Corollary}
\newtheorem*{rem*}{Remark}
\newtheorem{clm}[thm]{Claim}
\newtheorem*{clm*}{Claim}
\newtheorem{ques}[thm]{Question}
\theoremstyle{definition}
\newtheorem{defn}[thm]{Definition}
\usepackage[all]{xy}
\usepackage{tikz-cd}
\usepackage{mathtools}
\usepackage{comment}
\usepackage{authblk}
\usepackage{url}

\newcommand{\N}{\mathbb{N}}
\newcommand{\Z}{\mathbb{Z}}
\newcommand{\R}{\mathbb{R}}
\newcommand{\C}{\mathbb{C}}
\newcommand{\K}{\mathbb{K}}
\newcommand{\Span}{\mathrm{Span}}
\newcommand{\ve}{\varepsilon}
\makeatletter
\usepackage{etoolbox}

\usepackage{hyperref}
\hypersetup{
	colorlinks=true,
	linkcolor=red,
	citecolor=blue,
	pdftitle={Hyperfiniteness of the boundary actions of relatively hyperbolic groups},
	pdfpagemode=FullScreen,
}

\title{Decent actions of groups on restricted products}
\author{Chris Karpinski}

\begin{document}

\maketitle

\begin{abstract}
    An action of a group $G$ on a set $X$ is called ``decent'' if every subgroup of $G$ with a finite orbit in $X$ fixes a point in $X$ and every finitely generated subgroup of $G$ such that every element of the subgroup fixes a point of $X$ must itself have a global fixed point. In this article, we study conditions on when actions of groups on restricted products are ``decent''. We prove that the action of the automorphism group of a restricted product with base space the projective plane $\mathbb{P}^2(k)$ over a field $k$ is decent, generalizing a result of Lonjou--Przytycki--Urech. 
\end{abstract}

\section{Introduction}

This paper examines actions of groups with various ``fixed point'' properties. 

\begin{defn}
    Let $X_0$ be a set. An action of a group $G$ on $X_0$ is \textbf{elliptic} (respectively, \textbf{purely elliptic}) if $G$ has a global fixed point on $X_0$ (respectively, each element of $G$ fixes a point of $X_0$). An element of $G$ (respectively, a subgroup of $G$) is called \textbf{elliptic} if it fixes a point of $X_0$.
\end{defn}

A recent question of interest has been to identify and study actions of groups on spaces having the property that every subgroup of the group that locally fixes points must have a global fixed point. 

\begin{defn}
    \label{def: decent action}
    An action of a group $G$ on a set $X_0$ is \textbf{decent} if: 
    \begin{enumerate}
        \item Every subgroup of $G$ with a finite orbit on $X_0$ fixes a point. 
        \item Every \emph{finitely generated} subgroup of $G$ acting purely elliptically on $X_0$ fixes a point. 
    \end{enumerate}
\end{defn}

The terminology of ``decent actions'' was first introduced in \cite[Definition 1.2]{Algebraic}, where they appeared in the study of subgroups of the plane Cremona group $\mathrm{C}_2(k)$, which is the group of all birational transformations of the projective plane $\mathbb{P}^2(k)$ over a field $k$. However, such actions have been studied in the literature before their apparition in \cite{Algebraic} (see, for instance, \cite{Breuillard_Fujiwara}, \cite{Locallyelliptic},  \cite{NOP22}).  

Decent actions appear in many natural examples. For instance, any isometric action of a group on a tree is decent (see, for instance, \cite[Proposition 2.5]{Introduction_to_group_theory}). For $d =2, 3$, any isometric action of a group on $\R^d$ equipped with the Euclidean metric is decent (\cite[Proposition 9.2]{Breuillard_Fujiwara}). Also, by \cite[Chapter II.2, Corollary 2.8(1)]{BH99}, item 1 in Definition \ref{def: decent action} holds for any isometric action of a group on a complete CAT(0) space, however, item 2 in Definition \ref{def: decent action} may fail even for actions on Hilbert spaces (for every $n \geq 2$, examples of finitely generated groups of isometries of $\R^{2n}$ with the Euclidean metric acting purely elliptically but not elliptically on $\R^{2n}$ were constructed in \cite[Example 9.1]{Breuillard_Fujiwara}). 

Haettel and Osajda conjecture in \cite{Locallyelliptic} that every action of a group on a ``finite-dimensional non-positively curved complex'' is decent. See \cite[Definition 2.1]{Locallyelliptic} for the precise definition of finite-dimensional and non-positively curved complex. 

In this paper, we generalize results of \cite{Algebraic} on establishing decentness of actions of groups on \emph{restricted products} (see Definition \ref{restricted product} below). 

The main result of Lonjou--Przytycki--Urech (\cite{Algebraic}) that we aim to generalize is \cite[Theorem 1.10]{Algebraic} (see Section \ref{res prod} for the definition of $G^{\oplus}$ and $\oplus_{p \in P} (X_p, x_p)$ below). 

\begin{thm}(\cite[Theorem 1.10]{Algebraic})
    \label{thm: Algebraic main theorem}
    Let $G_0$ be a group acting decently on a set $X_0$, with basepoint $x_0$. Let $P$ be the projective line $\mathbb{P}^1$ over an algebraically closed field $k$, and let $H = Aut(P) = PGL_2(k)$. Then $G^{\oplus}$ acts decently on $\oplus_{p \in P} (X_p, x_p)$, with $X_p = X_0$ and $x_p = x_0$ for all $p \in P$.  
\end{thm}

Note that by \cite[Lemma 4.2]{Algebraic}, item 1 in Definition \ref{def: decent action} for $G^{\oplus} \curvearrowright \oplus_{p \in P} (X_p, x_p)$ follows from item 1 for $G_0 \curvearrowright X_0$. Therefore, it suffices to prove item 2 for $G^{\oplus} \curvearrowright \oplus_{p \in P} (X_p, x_p)$, i.e.\ that each finitely generated subgroup $\Gamma < G^{\oplus}$ acting purely elliptically on $\oplus_{p \in P} (X_p, x_p)$ acts elliptically. 

Also, note that while stated for algebraically closed fields $k$, Theorem \ref{thm: Algebraic main theorem} holds for arbitrary fields. Indeed, if $K$ is an arbitrary field and $k$ is an algebraic closure of $K$, then denoting $P = \mathbb{P}^1(K)$, $P' = \mathbb{P}^1(k)$, $G^{\oplus} = \prod_P^r G_0 \rtimes PGL_2(K)$, $G'^{\oplus} = \prod_{P'}^r G_0 \rtimes PGL_2(k)$, $X = \oplus_{p \in P} (X_p, x_p)$ and $X'=\oplus_{p \in P'} (X_p, x_p)$, we can identify $P < P'$ and we have natural embeddings:

$$\phi:X \hookrightarrow X'$$

and

$$\psi:G^{\oplus} \hookrightarrow G'^{\oplus}$$

defined by $\phi(x) = z$ where $z_p = x_p$ for each $p \in P$ and $z_q = x_0$ for each $q \in P' \setminus P$, and $\psi((g_p)_{p \in P}, h) = ((g_q')_{q \in P'},h)$ where $g'_p = g_p$ for each $p \in P$ and $g'_q = 1$ for each $q \in P' \setminus P$, with $\phi$ being $\psi$-equivariant.

Thus, if $\Gamma < G^{\oplus}$ is finitely generated and acts purely elliptically, then the same is true for $\psi(\Gamma)$, and hence by Theorem \ref{thm: Algebraic main theorem}, we have that $\psi(\Gamma)$ fixes a point $y \in X'$. But since $G < PGL_2(K)$, we can choose $y \in \phi(X)$. Indeed, given a fixed point $y'$ of $\psi(\Gamma)$ in $X$, we define $y_p = y'_p$ for each $p \in P$ and $y_p = x_0$ for $p \in P' \setminus P$. Then $y \in \phi(X)$ and $y$ is still a fixed point of $\psi(\Gamma)$. Indeed, given $\gamma \in \Gamma$, write $\psi(\gamma) = ((g_p')_{p \in P'}, h)$ where $g_p' = 1$ for each $p \in P' \setminus P$. Then, for $p \in P' \setminus P$, since $G$ preserves $P' \setminus P$ we have 

$$(\psi(\gamma)y)_p = y_{h^{-1}(p)} = x_0 = y_p$$  and for each $p \in P$, since $y'$ is a fixed point of $\psi(\Gamma)$, we have have 

$$(\psi(\gamma)y)_p = g_p' y_{h^{-1}(p)} = g_p' y'_{h^{-1}(p)} = y_p' = y_p$$

Hence, by $\psi$-equivariance of $\phi$, we have that $\Gamma$ fixes a point of $X$.

We generalize Theorem \ref{thm: Algebraic main theorem} to the following, which is the main result of this paper: 

\begin{thm}
    \label{thm: main thm}
    Let $G_0$ be a group acting decently on a set $X_0$, with basepoint $x_0$. Let $P$ be the projective plane $\mathbb{P}^2$ over a field $k$, and let $H = Aut(P) = PGL_3(k)$. Then $G^{\oplus}$ acts decently on $\oplus_{p \in P} (X_p, x_p)$, with $X_p = X_0$ and $x_p=x_0$ for all $p \in P$.  
\end{thm}

For the same reason as noted for Theorem \ref{thm: Algebraic main theorem}, item 1 in Definition \ref{def: decent action} for $G^{\oplus} \curvearrowright \oplus_{p \in P} (X_p, x_p)$ follows automatically from item 1 for $G_0 \curvearrowright X_0$, and it suffices to prove Theorem \ref{thm: main thm} for algebraically closed fields, hence we will assume that $k$ is algebraically closed in the proof of Theorem \ref{thm: main thm}. 

As an application of Theorem \ref{thm: Algebraic main theorem}, Lonjou--Przytycki--Urech are able to deduce local to global results on subgroups of $\mathrm{C}_2(k)$, most notably \cite[Theorem 1.1]{Algebraic}, stating that every finitely generated subgroup of $\mathrm{C}_2(k)$ all of whose elements are algebraic is bounded (see \cite{Algebraic} for the definitions of ``algebraic'' elements and ``bounded'' subgroups of $\mathrm{C}_2(k)$). This is achieved by studying the action of $\mathrm{C}_2(k)$ on a restricted product of simplicial trees called the ``Jonqui\`eres complex'', and showing that this action is decent. It is then shown that algebraic elements of $\mathrm{C}_2(k)$ are precisely the elliptic elements with respect to this action, and that the bounded subgroups are precisely the elliptic subgroups with respect to this action. 

An important ingredient in the proof of Theorem \ref{thm: Algebraic main theorem} is to establish decentness of actions on restricted products where the acting subgroup has abelian projection to $PGL_2(k)$. See Section \ref{res prod} for an explanation of the notation and terminology below. 

\begin{lem}(\cite[Lemma 4.7]{Algebraic})
    Let $G_0$ be a group acting decently on $X_0$. Let $H$ be a group acting on $P$, and let $\Gamma < G^{\oplus}$ be finitely generated and act purely elliptically on $X = X_0^{\oplus P}$. Suppose that $\mathrm{proj}_H(\Gamma)$ is abelian. If $\mathrm{proj}_H(\Gamma) < H$ is trivial or it contains an element that has only finitely many finite orbits on $P$, then $\Gamma$ fixes a point of $X$.
\end{lem}

We generalize this to the following:

\begin{lem}
    \label{lem: virtually nilpotent projection}
    Let $G_0$ be a group acting decently on $X_0$. Let $H$ be a group acting on $P$, and let $\Gamma < G^{\oplus}$ be finitely generated and act purely elliptically on $X$. If $\mathrm{proj}_H(\Gamma)$ is virtually nilpotent, then $\Gamma$ fixes a point of $X$.
\end{lem}

Note that we are able to remove the assumption of the existence of an element of $\mathrm{proj}_H(\Gamma)$ with only finitely many finite orbits on $P$. 

In another direction, we prove the following:

\begin{lem}
        Let $G_0$ be a group acting decently on $X_0$. Let $H$ be a group acting on $P$, and let $\Gamma < G^{\oplus}$ be finitely generated and acting purely elliptically  on $X$. 
        
        If $\mathrm{proj}_H(\Gamma) < H$ is trivial or has an abelian normal subgroup that contains an element with only finitely many finite orbits on $P$, then $\Gamma$ fixes a point of $X$. 
    \end{lem}

    As a corollary, we deduce the following for solvable projections:

    \begin{cor}
        \label{solvable lemma}
        If $\mathrm{proj}_H(\Gamma)$ has a normal series $G_1 < G_2 < \ldots < G_n = \mathrm{proj}_H(\Gamma)$ such that each $G_i/G_{i-1}$ is abelian and such that there exists $t \in G_1$ with only finitely many finite orbits on $P$, then $\Gamma$ fixes a point of $X$.
    \end{cor}

    We still do not know if Lemma \ref{lem: virtually nilpotent projection} generalizes to the case of when $\mathrm{proj}_H(\Gamma)$ is virtually solvable for arbitrary groups $H$ acting on arbitrary sets $P$. However, for the case of when $P = \mathbb{P}^2(k)$ over a field $k$ and $H = PGL_3(k) \curvearrowright P$, we are able to prove the following: 

    \begin{lem}
    \label{lem: virtually solvable projection}
    Let $G_0$ be a group acting decently on $X_0$. Let $P = \mathbb{P}^2(k)$ over a field $k$ and $H = PGL_3(k) \curvearrowright P$, and let $\Gamma < G^{\oplus}$ be finitely generated and act purely elliptically on $X$. If $\mathrm{proj}_H(\Gamma)$ is virtually solvable, then $\Gamma$ fixes a point of $X$.
\end{lem}

We deduce Lemma \ref{lem: virtually solvable projection} as a corollary of the following more general condition on $\mathrm{proj}_H(\Gamma)$ acting on $P$ and Borel's theorem (see, for instance, \cite[14.65]{KapDrutu}). 

\begin{lem}
    \label{lem: virtually fixed point}
    Let $G_0$ be a group acting decently on $X_0$. Let $P = \mathbb{P}^2(k)$ over a field $k$ and $H = PGL_3(k) \curvearrowright P$, and let $\Gamma < G^{\oplus}$ be finitely generated and act purely elliptically on $X$. If $\mathrm{proj}_H(\Gamma)$ has a finite index subgroup that fixes a point of $P$, then $\Gamma$ fixes a point of $X$.
\end{lem}

\subsection{Organization of the paper}

This paper is organized into two main sections. In the first section, we cover the necessary background and preliminary notions. In particular, in Section \ref{res prod} we define \emph{restricted products} and their automorphism groups and in Section \ref{biregularity}, we discuss the concepts of \emph{biregularity} and \emph{singularity}. Then in Section \ref{Top dynamics Pn}, we review the topology and metric structure of the projective space $\mathbb{P}^n(\K)$ over a normed field $\K$, and define the notion of \emph{proximal} and \emph{very proximal} automorphisms of projective space. We will need to use these notions in Section \ref{sec: NSD element}.  

We prove Theorem \ref{thm: main thm} in Section \ref{sec: proof of main theorem}. The proof of Theorem \ref{thm: main thm} is divided into two main cases depending on whether the projection $G$ of the acting group $\Gamma$ onto $PGL_3 (k)$ has a very proximal element or not. The case of the existence of a very proximal element is handled in Section \ref{sec: NSD element} and in the case of no very proximal elements, we prove that $G$ must have a finite index subgroup which fixes a point of $\mathbb{P}^2(\K)$ for some extension $\K$ of a finitely generated subfield $K$ of $k$. This latter case is handled in Section \ref{section: fixed subspace}, where we prove Lemma \ref{lem: virtually fixed point}. 

\subsection{Acknowledgements} I am thankful to Piotr Przytycki for introducing me to the subject of this paper and for his unwavering help and guidance throughout the duration of the project. I thank also Christian Urech for carefully listening to the arguments in this paper and for suggesting possible applications to the proof of Theorem \ref{thm: main thm}. I would also like to thank Yves Cornulier for providing the proof of Lemma \ref{lem: decent actions nilpotent}.

\section{Preliminaries}

\subsection{Restricted products}
\label{res prod}

\begin{defn}
    \label{restricted product}
    Let $P$ be a set. For each $p \in P$, let $X_p$ be a set with a base point $x_p$.  The \textbf{restricted product of $(X_p, x_p)$} is the set: 

    $$X = \oplus_{p \in P}(X_p, x_p) := \{(y_p)_p \in \prod_P X_p : y_p = x_p \text{ for all but finitely many } p\}$$
\end{defn}

In this paper, we will specialize to the case of when the based spaces $(X_p, x_p)$ are all one fixed based space $(X_0, x_0)$. In this case, we will denote $X_0^{\oplus P} = \oplus_{p \in P}(X_0, x_0)$ when we do not wish to distinguish the basepoint $x_0$, and we will also sometimes use the notation $\prod^r_P X_0$ for the restricted product $X_0^{\oplus P}$.

Suppose that $G_0, H$ are groups such that  $H \curvearrowright P$ and $G_0 \curvearrowright X_0$. Then the (unrestricted) wreath product: 

    $$G_0 \wr_P H := \prod_P G_0 \rtimes H$$

    acts on $\prod_P X_0$ via, for each $(x_p)_p \in \prod_P X_0$ and each $g \in \prod_P G_0, h \in H$: 

    $$(ghx)_p = g_p x_{h^{-1}p}$$

    Denote $G^{\oplus}$ the subgroup of $G_0 \wr_P H$ preserving $X$. We will also denote $$\prod_P^r G_0 = \{(g_p)_p \in \prod_P G_0 : g_p \in \mathrm{Stab}_{G_0}(x_0) \ \text{for all but finitely many $p \in P$}\}$$

Note that for any $\Gamma < G_0 \wr_P H$, we have a group homomorphism $\psi: \Gamma \to H$ given by projecting onto the $H$ factor in the semi-direct product $\prod_P G_0 \rtimes H$. We will denote the image of $\Gamma$ in $H$ by $\mathrm{proj}_H(\Gamma)$. 

Similarly, note that any $\Gamma < G_0 \wr_P H$ acts on $P$ via the $H$ coordinate. For any $p \in P$, the projection $\mathrm{Stab}_{\Gamma}(p) \to G_0$ of the stabilizer in $\Gamma$ of $p$ onto the $p$th coordinate in the product $\prod_P G_0$ is a group homomorphism. 

We record the following lemma, which we will use in Section \ref{sec: NSD element} and Section \ref{sec: virtually nilpotent case}. 

\begin{lem}
    \label{Inf orb fix pt}

    Let $t \in G^{\oplus}$ and let $O$ be an infinite orbit of $t$ on $P$. Then for any fixed points $y,z$ of $t$ in $X$, we have $y_p = z_p$ for each $p \in O$. 
    
\end{lem}

\begin{proof}
    Let $z$ be a fixed point of $t = gh$ in $X$ for $g \in \prod_P G_0$ and $h \in H$. We will show that the coordinates of $z$ in $O$ are uniquely determined. 

    Fixing $p \in O$, we can enumerate the elements of $O$ as $O = \{t^n p : n \in \Z\}$. Since $O$ is infinite, the map $\Z \ni n \mapsto t^n p$ is a bijection, we can identify $O$ with $\Z$ and write the $t$-action on $O$ as $t^n i = i + n$ for any $i,n \in \Z$. We will also write $z_{t^n p} (= z_{h^n p})$ as $z_n$ and $g_{t^n p} (= g_{h^n p})$ as $g_n$.

    Since $t(z) = z$, we have for each $n \in \Z$ that $g_n z_{n-1} = z_n$. Denote $S = \mathrm{Stab}_{G_0}(x_0)$. Since $t \in G^{\oplus}$, there exists $N \leq M$ in $\Z$ such that $g_n \in S$ for all $n \leq N$ and for all $n \geq M$. Up to re-indexing $O$, we can assume that $N = 0$. 

    We claim that $z_n = x_0$ for all $n \leq 0$. Indeed, since $z \in X$, there exists $K \leq 0$ such that for all $n \leq K$, $z_n = x_0$. By induction, for each $K \leq n \leq 0$, we have that 
    
    $$z_n = g_nz_{n-1} = g_n g_{n-1} z_{n-2} = \cdots = g_n g_{n-1} \cdots g_K z_K$$
    
    Since $g_i \in S$ for all $i \leq 0$, the latter expression is in $S \{z_K\} = S \{x_0\} = \{x_0\}$. Therefore, $z_n = x_0$ for all $K \leq n \leq 0$, and hence for all $n \leq 0$ (since $K$ was chosen such that $z_n = x_0$ for all $n \leq K$). Note that a symmetric argument yields that $z_n = x_0$ for all $n \geq M$, but we will not need this.

    Now for each $n > 0$, by induction we have: 

    $$z_n = g_n z_{n-1} = g_n g_{n-1} z_{n-2} = \cdots = g_n g_{n-1} \cdots g_1 z_0 = g_n g_{n-1} \cdots g_1 x_0$$

    Thus, the coordinates of $z$ on $O$ depend only on $g$ and $x_0$, hence are uniquely determined.

\end{proof}

\subsection{Biregularity}
\label{biregularity}

We review the notion of biregularity of maps as defined in \cite[Section 4.1]{Algebraic}.

Let $(X_0,x_0)$ be a based set and let $X$ be the restricted product of $(X_0, x_0)$ over a set $P$. Let $G_0 \curvearrowright X_0$ and $H \curvearrowright P$, and $G^{\oplus}$ the subgroup of $G_0 \wr_P H$ preserving $X$. 

\begin{defn}
    Fix $z \in X$. We say that $f \in G^{\oplus}$ is \textbf{biregular} over $p \in P$ with respect to $z$ if $f(z)_{f(p)} = z_{f(p)}$. Equivalently, for $f = (g,h)$ with $g \in \prod_P G_0$ and $h \in H$, we have $g_p(z_{h^{-1}(p)}) = z_p$. Otherwise, we say that $f$ is \textbf{singular} over $p$ with respect to $z$. 

    An element $f \in G^{\oplus}$ has \textbf{persistent fibre} over $p \in P$ if there exists $l \geq 1$ such that for all $n \geq l$, we have that $f^n$ is singular over $p$ with respect to $z$ and $f^{-n}$ is biregular over $p$ with respect to $z$.
    
\end{defn}

\begin{rem}(\cite[Remark 4.5]{Algebraic})
    A map $f \in G^{\oplus}$ fixes a point $z \in X$ if and only if $f$ is biregular over each $p \in P$ with respect to $z$. 
\end{rem}

\begin{rem}(\cite[Remark 4.6]{Algebraic})
    If $f \in G^{\oplus}$ has persistent fibre over $p \in P$, then it does not fix a point of $X$. 
\end{rem}

\subsection{Topology and dynamics on $\mathbb{P}^n$}
\label{Top dynamics Pn}

In this section, we review the canonical topological and metric structure of the projective space $P = \mathbb{P}^n(\K)$ over a normed field $\K$ and study various dynamical behaviour of automorphisms of $\mathbb{P}^n$, introducing the notions of \emph{proximal} and \emph{very proximal} elements. We follow \cite[Section 2.9]{KapDrutu}. 

We will typically work with \emph{local} fields $\K$, that is, fields equipped with a norm whose associated metric induces a locally compact topology on $\K$. The norm on $\K$ induces a product norm on $V = \K^{n+1}$. If $\K$ is archimedean (i.e.\ $\{\vert n \cdot 1_{\K} \vert : n \in \N\}$ is unbounded, equivalently $\K = \R$ or $\K = \C$), then we equip $V$ with the $\ell^2$ norm: 

$$\vert (x_1,\ldots,x_{n+1}) \vert = \sqrt{\vert x_1\vert^2 + \cdots + \vert x_{n+1} \vert^2}$$

For each $x = (x_1,\ldots,x_{n+1}) \in V$. 

Otherwise, if $\K$ is non-archimedean (i.e.\ $\{\vert n \cdot 1_{\K} \vert: n \in \N\}$ is bounded), then we equip $\K$ with the $\ell^{\infty}$ norm: 

$$\vert (x_1,\ldots,x_{n+1}) \vert = \max \{\vert x_1 \vert, \ldots, \vert x_{n+1} \vert\}$$

for each $x = (x_1,\ldots,x_{n+1}) \in V$. 

If $\K$ is archimedean, we equip the wedge product $V \wedge V: = (V \otimes V)/\Span(v \otimes w - w \otimes v : v,w \in V)$ with the Euclidean norm by declaring the basis vectors $e_i \wedge e_j$ to be orthonormal, so that: 

$$\vert v \wedge w \vert^2 = \vert v \vert^2 \vert w \vert^2 - \vert \langle v,w \rangle \vert^2 $$

for each $v,w \in V$, where $\langle \cdot, \cdot \rangle$ denotes the standard Euclidean/Hermitian inner product on $V$. 

If $\K$ is non-archimedean, we equip $V \wedge V$ with the max norm, so that: 

$$\vert v \wedge w \vert = \max_{i<j} \{\vert x_i y_j - x_j y_i \vert \}$$

for $v = (x_1,\ldots,x_{n+1}), w=(y_1,\ldots,y_{n+ 1})$. 

We then define the following distance function on $\mathbb{P}^n$, denoting $[v]$ the equivalence class in $\mathbb{P}^n = \mathbb{P}(V)$ of a vector $v \in V$: 

\begin{defn}
    The \textbf{chordal metric} on $\mathbb{P}^n = \mathbb{P}(V)$ is defined by: 

    $$d([v], [w]) = \frac{\vert v \wedge w \vert}{\vert v \vert \vert w \vert}$$
\end{defn}

By \cite[Theorem 2.77]{KapDrutu}, $d$ is a metric on $P$ whose associated topology is the quotient topology on $P$ induced from $V \setminus \{0\}$ and the quotient map $V \setminus \{0\} \ni v \mapsto [v]$ (see \cite[Exercise 2.78]{KapDrutu}). We will not need to explicitly use this metric, however, we still record its definition for the sake of completeness. 

We now discuss a class of well-behaved automorphisms of $P$, the \emph{proximal} and \emph{very proximal} automorphisms. Recall that every $g \in GL(V)$ induces a projective automorphism of $P$, and that the group of projective automorphisms is the projective linear group $Aut(P) = PGL(V) = GL(V) / \langle \lambda I : \lambda \in \K\setminus \{0\}\rangle$. 

Given $[g] \in PGL(V)$, we will frequently abuse notation and identify $[g]$ with its lift $g \in GL(V)$. 

\begin{defn}
    Let $\K$ be a normed field. An element $t \in PGL_n(\K)$ is called \textbf{proximal} if $t$ has a unique eigenvalue $\lambda_+$ of algebraic multiplicity 1 with maximum absolute value, called the \textbf{dominant eigenvalue} of $t$. An element $t \in PGL_n(k)$ is \textbf{very proximal} if $t$ and $t^{-1}$ are both proximal, in which case $t$ has unique maximal and minimal absolute value eigenvalues $\lambda_+$ and $\lambda_-$ (both of algebraic multiplicity 1).  
\end{defn}




\section{Proof of Theorem \ref{thm: main thm}}
\label{sec: proof of main theorem}


\subsection{The case of a fixed point in $P = \mathbb{P}^2(k)$}

\label{section: fixed subspace}

Suppose that $G_0 \curvearrowright X_0$ is decent, $P = \mathbb{P}^2(k)$ over an arbitrary field $k$ and $H = PGL_3(k)$. Let $\Gamma < G^{\oplus}$ be finitely generated and act purely elliptically on $X = \prod^r_P X_0$. Denote $G = \mathrm{proj}_H(\Gamma)$. In this section we prove Theorem \ref{thm: main thm} in the case that $G$ virtually fixes a point of $P$. 

\begin{lem}
    \label{lem: virtual fixed point on P implies elliptic}
    Suppose that $G$ has a finite index subgroup that fixes a point of $P$. Then $\Gamma \curvearrowright X$ is elliptic. 
\end{lem}

\begin{proof}
    Since a finite index subgroup of $G$ lifts to a finite index subgroup of $\Gamma$, and having a finite index elliptic subgroup $\Gamma$ implies that $\Gamma$ acts elliptically by \cite[Remark 4.2]{Algebraic}, we can assume that $G$ fixes a point $u$ of $P$. 

    Since $u$ is $G$-invariant we have that $\Gamma \curvearrowright X_0^{\oplus u}$ and acts elliptically. Indeed, since $G$ fixes $u$, we have that $G$ projects to a finitely generated subgroup $G_u$ of $G_0$ which acts purely elliptically on $X_0$ (since $\Gamma \curvearrowright X$ is purely elliptic). Since $G_0 \curvearrowright X_0$ is decent, we obtain that $G_u$ fixes a point $z_0$ of $X_0$. Then $\Gamma \curvearrowright X_0^{\oplus u}$ fixes the point $(z_0)_u$. Therefore, it suffices to show that $\Gamma$ fixes a point of $X_0^{\oplus P \setminus u}$. 

    Denoting $e_0 = [1:0]$, we can embed: 

    $$\iota: P \setminus u \hookrightarrow L \times \mathbb{P}^1 \setminus e_0$$

    where $L$ is the set of all projective lines in $P$ containing $u$, as follows. Fixing a basis $(u, b_l)$ for each line $l \in L$, we obtain a bijection $B_l : l \to \mathbb{P}^1$. Note that $G$ acts on $L$, since $G$ fixes $u$. 

    Now, given $x \in P \setminus u$, we have that $x$ lies on the line $l = \Span( u, x )$ through $u$. Map $x$ to the tuple $\iota(x):=(l, B_{l}(x))) \in L \times \mathbb{P}^1\setminus e_0$. Then the map $\iota$ is injective since if $\iota(x) = \iota(y)$ for $x,y \in P \setminus u$, then $x$ and $y$ are on the same line through $u$ and have the same coordinates on this line with respect to the basis $(u,b_l)$, hence $x=y$. 

    This yields an embedding (in fact, a bijection): 

    $$\phi : X_0^{\oplus P \setminus u} \hookrightarrow \prod_L^r (\prod_{\mathbb{P}^1\setminus e_0}^r X_0 ) $$

    defined by $(x_p)_{p \in P \setminus u} \mapsto (\vec{x}_l)_{l \in L}$ where for each $l \in L$, $\vec{x}_l \in  \prod_{\mathbb{P}^1\setminus e_0}^r X_0$ is such that for each $r \in \mathbb{P}^1 \setminus e_0$, $(\vec{x}_l)_r = x_p$, where $p \in P \setminus u$ is the point on the line $l \in L$ with $B_l(p) = r$. Here, taking $x_0$ as the basepoint of $X_0$, we take the constant sequence $\vec{x} = (x_0, x_0, \ldots)$ as the basepoint for $\prod_{\mathbb{P}^1\setminus e_0}^r X_0$ and the constant sequence $(\vec{x}, \vec{x}, \ldots)$ as the basepoint for $\prod_{L}^r \prod_{\mathbb{P}^1\setminus e_0}^r X_0$. We have that $\phi$ is injective since $\iota$ is injective. 

    We next construct an injective homomorphism $\psi: \Gamma \hookrightarrow \prod_L^{r}  (\prod_{\mathbb{P}^1\setminus e_0}^rG_0 \rtimes \mathrm{Stab}_{PGL_2(k)}(e_0)) \rtimes  G$ defined by: 

    $$((g_p)_{p \in P \setminus u}, h) \mapsto ((\vec{g_l}), h_l)_{l \in L}, h)$$

    where for each $l \in L$, $\vec{g}_l \in \prod_{\mathbb{P}^1\setminus e_0}^r X_0$ is such that for each $r \in \mathbb{P}^1 \setminus e_0$, $(\vec{g}_l)_r = g_p$, where $p \in P \setminus u$ is the point on the line $l$ with $B_l(p) = r$. For each $l \in L$, we define $h_l \in \mathrm{Stab}_{PGL_2(k)}(e_0)$ as follows. Put $f_l$ to be the map from $h^{-1}(l) \to l$ induced by $h$. Then define $h_l = B_l f_l B_{h^{-1}(l)}^{-1} \in PGL_2(k)$. We have that $h_l(e_0) = e_0$ since $h(u) = u$, and $B_l(u) = e_0 = B_{h^{-1}(l)}(u)$

    Then $\psi$ is injective since it is injective on $\prod_{P \setminus u}^{r} G_0$ due to the injectivity of $\iota$, and it is injective on $G$ since it is the identity on the $G$ component. We now show that $\psi$ is a group homomorphism. 

    Consider two elements $\gamma = ((g_p), h)$ and $\gamma'=((g_p'), h')$ of $\Gamma$. Write $\psi(\gamma) = ((\vec{g_l}), h_l)_{l \in L}, h)$ and $\psi(\gamma') = ((\vec{g_l}'), h'_l)_{l \in L}, h')$. We have $\gamma \gamma' = ((g_p g'_{h^{-1}(p)})_p, hh')$, and 
    $$\psi(\gamma \gamma') = (((\vec{a_l}), (hh')_{l})_{l}, hh')$$ where for each $l \in L$ and each $r \in \mathbb{P}^1 \setminus e_0$, we have $(\vec{\alpha}_l)_r = g_p g_{h^{-1}(p)}'$ where $p \in l$ is such that $B_{l}(p) = r$.

    Now multiplying: 

    \begin{align*}
        \psi(\gamma)\psi(\gamma') &= ((\vec{g_l}), h_l)_{l \in L}, h)((\vec{g_l}'), h'_l)_{l \in L}, h')\\
        &= (((\vec{g_l})h_l \cdot(\vec{g}'_{h^{-1}(l)}), h_l h'_{h^{-1}(l)})_{l \in L}, hh')
    \end{align*}

    Now, to check that $\psi$ is a homomorphism, we need to check that: 

    \begin{enumerate}
        \item For each $l \in L$, $(hh')_l = h_lh_{h^{-1}(l)}'$
        \item For each $l \in L$, $\vec{a}_l = (\vec{g_l})h_l \cdot(\vec{g}'_{h^{-1}(l)})$
    \end{enumerate}

    For 1, let $F$ denote the map $(hh')^{-1}(l) \to l$ induced by $hh'$. Since $(hh')^{-1}(l) = (h')^{-1}h^{-1}(l)$, we have that $F = f_l f'_{h^{-1}(l)}$. Thus,

    \begin{align*}
        (hh')_l &= B_l  f_l f'_{h^{-1}(l)} B_{(hh')^{-1}(l)}^{-1} \\
        &= B_l f_l f'_{h^{-1}(l)} B_{(h')^{-1}h^{-1}(l)}^{-1} \\
    \end{align*}

    On the other hand, 

    \begin{align*}
        h_lh_{h^{-1}(l)}' &= B_l f_l B_{h^{-1}(l)}^{-1} B_{h^{-1}(l)} f'_{h^{-1}(l)} B_{(h')^{-1}h^{-1}(l)}^{-1} \\
        &= B_l f_l f'_{h^{-1}(l)}B_{(h')^{-1}h^{-1}(l)}^{-1} 
    \end{align*}

    Thus, we see that $(hh')_l = h_lh_{h^{-1}(l)}'$. 

    Next, we check 2.

    For each $l \in L$ and each $r \neq e_0$, we have: 

    $$((\vec{g}_l)h_l \cdot(\vec{g}'_{h^{-1}(l)}))_r = g_p g'_q$$
    where $p \in l$ is such that $B_l(p) = r$, i.e.\ $p = B_l^{-1}(r)$ and $q \in h^{-1}(l)$ is such that $B_{h^{-1}(l)}(q) = h_l^{-1}(r)$, i.e.\ $q = B_{h^{-1}(l)}^{-1}(h_l^{-1}(r))$. We need to show that $q = h^{-1}(p)$, i.e.\ $B_{h^{-1}(l)}^{-1}(h_l^{-1}(r)) = h^{-1}B_l^{-1}(r)$. By definition of $h_{l}$, we have:

    $$h_l = B_lhB_{h^{-1}(l)}^{-1}$$

    Re-arranging, we obtain: 

    $$B_{h^{-1}(l)}^{-1}h_l^{-1} = h^{-1}B_l^{-1}$$

    as desired. We conclude that 2 holds. 

    Thus, we conclude that $\psi$ is a group homomorphism. 

    We next check that the map $\phi$ is $\psi$-equivariant, i.e.\ that $$\phi(\gamma x) = \psi(\gamma) \phi(x)$$ for each $\gamma \in \Gamma$ and $x \in X_0^{\oplus P \setminus u}$.

    Write $\gamma = ((g_p)_p, h)$ and $x = (x_p)_p$, with $\psi(\gamma) = ((\vec{g_l}, h_l)_l, h)$ and $\phi(x) = (\vec{x_l})_l$. 

    On the one hand, we have: 

    $$\gamma x = (g_p x_{h^{-1}(p)})_p$$

    So that 

    $$\phi(\gamma x) = (\vec{y}_{l})_l$$

    with $(\vec{y}_{l})_r = g_p x_{h^{-1}(p)}$ for each $l \in L$, $r \neq e_0$, and $p \in l$ such that $B_{l}(p) = r$.

    On the other hand, we have: 

    \begin{align*}
        \psi(\gamma) \phi(x) &= ((\vec{g_l}, h_l)_l, h)(\vec{x_{l}})_l \\
        &= ((\vec{g}_l)_l h_l \cdot(\vec{x}_{h^{-1}(l)}))_l
    \end{align*}

    Let $l \in L$ and $r \in \mathbb{P}^1 \setminus e_0$. Then:

    \begin{align*}
        ((\vec{g}_l)h_l (\vec{x}_{h^{-1}(l)}))_r 
        &= g_p x_q
    \end{align*}

    where $p \in l$ is such that $B_{l}(p) = r$ and $q \in h^{-1}(l)$ is such that $B_{h^{-1}(l)}(q) = h_{l}^{-1}(r)$. By above, using the definition of $h_{l}$, we have shown that $q=h^{-1}(p)$, hence we obtain: 

    $$(((\vec{g}_l), h_{l}) (\vec{x}_{h^{-1}(l)}))_r = g_p x_{h^{-1}(p)} = (\vec{y}_{l})_r$$

    We conclude that $\phi$ is $\psi$-equivariant. 

    To finish the proof, we will show that $\psi(\Gamma)$ fixes a point of $X':=\prod_L^r (\prod_{\mathbb{P}^1 \setminus e_0}^r X_0)$. 

    Since $\Gamma$ is finitely generated, so is $\psi(\Gamma)$ and since $\Gamma$ acts purely elliptically on $X_0^{\oplus P \setminus u}$ and $\phi$ is $\psi$-equivariant, we have that $\psi(\Gamma)$ also acts purely elliptically on $X'$. By Theorem \ref{thm: Algebraic main theorem}, since $G_0 \curvearrowright X_0$ is decent, we have that $\prod_{\mathbb{P}^1}^rG_0 \rtimes PGL_2(k)$ acts decently on $\prod_{\mathbb{P}^1}^rX_0$ (as remarked in the introduction, Theorem \ref{thm: Algebraic main theorem} holds for arbitrary fields). We need to show that $\prod_{\mathbb{P}^1 \setminus e_0}^rG_0 \rtimes \mathrm{Stab}_{PGL_2(k)}(e_0)$ acts decently on $\prod_{\mathbb{P}^1 \setminus e_0}^rX_0$. 
    
    Let $\Lambda < \prod_{\mathbb{P}^1 \setminus e_0}^rG_0 \rtimes \mathrm{Stab}_{PGL_2(k)}(e_0)$ be finitely generated and act purely elliptically. Then $\Lambda$ embeds into $ \prod_{\mathbb{P}^1}^rG_0 \rtimes PGL_2(k)$ by mapping $((\lambda)_{r \in \mathbb{P}^1 \setminus e_0}, h)$ to $((\lambda')_{r \in \mathbb{P}^1}, h)$ with $\lambda'_r = \lambda_r$ for $r \neq e_0$ and $\lambda'_{e_0} = 1$, and mapping $h$ to itself under the inclusion $\mathrm{Stab}_{PGL_2(k)}(e_0) \hookrightarrow PGL_2(k)$. Since $\Lambda$ is finitely generated and acts purely elliptically, so is the image $\Lambda'$. We also have an embedding $\prod_{\mathbb{P}^1 \setminus e_0}^r X_0 \hookrightarrow \prod_{\mathbb{P}^1}^r X_0$ via $(x_r)_{r \in \mathbb{P}^1 \setminus e_0} \mapsto (x'_r)_{r \in \mathbb{P}^1}$ by putting $x'_r = x_r$ for $r \neq e_0$ and $x'_{e_0} = x_0$. It is not difficult to see that this embedding is equivariant with respect to the above embedding of $\Lambda$ into $\prod_{\mathbb{P}^1}^rG_0 \rtimes PGL_2(k)$. By Theorem \ref{thm: Algebraic main theorem}, we have that $\Lambda'$ fixes a point $y$ of $\prod_{\mathbb{P}^1}^r X_0$, from which it follows that $\Lambda$ fixes a point $z'$ of $X'$ defined by $(z')_r = y_r$ for each $r \in \mathbb{P}^1 \setminus e_0$.
    
    Finally, noting that $L \cong \mathbb{P}^1$ one last application of Theorem \ref{thm: Algebraic main theorem} yields that $\prod_L^{r}  (\prod_{\mathbb{P}^1 \setminus e_0}^rG_0 \rtimes \mathrm{Stab}_{PGL_2(k)}(e_0)) \rtimes  G$ acts decently on $\prod_L^r ( \prod_{\mathbb{P}^1 \setminus e_0}^r X_0) = X'$. Therefore, $\psi(\Gamma)$ fixes a point $x' \in X'$. Since $\phi$ is a $\psi$-equivariant bijection, $\Gamma$ fixes the point $x = \phi^{-1}(x')$ of $X_0^{\oplus P \setminus u}$. By our discussion at the beginning of the proof, we thus conclude that $\Gamma$ fixes a point of $X$. 
\end{proof} 

As a corollary, we obtain the following.

\begin{cor}
    \label{cor: virtually solvable implies elliptic}
    If $G$ is virtually solvable, then $G$ has a finite index subgroup fixing a point of $P$. In particular, $\Gamma$ is elliptic. 
\end{cor}

\begin{proof}
    First, we will show that if $G$ is virtually solvable, then $G$ contains a finite index subgroup conjugate into the group $T_3(k)$ of invertible upper triangular matrices. Indeed, let $H < G$ be a finite index solvable subgroup. Let $\overline{H}$ denote the Zariski closure of $H$. The irreducible component of the identity $\overline{H}^o$ has finite index in $\overline{H}$ and is connected (see \cite[Proposition 5.92]{KapDrutu}). We have that $\overline{H}$ is solvable, and so $\overline{H}^o$ is also solvable. By \cite[Theorem 14.65]{KapDrutu}, we have that $\overline{H}^o$ is conjugate into $T_3(k)$. Let $H^o = \overline{H}^o \cap H$. Then $H^o$ is conjugate into $T_3(k)$ and has finite index in $H$, hence also finite index in $G$. Since $H^o$ is conjugate into $T_3(k)$, we have that $H^o$ fixes a point of $P$. By Lemma \ref{lem: virtual fixed point on P implies elliptic}, $\psi^{-1}(H^o)$ acts elliptically on $X$, and hence so does $\Gamma$ (since $\vert \Gamma : \psi^{-1}(H^o) \vert < \infty$).
    
\end{proof}

\subsection{The case of a very proximal element}
\label{sec: NSD element}

Keeping the notation $H,P,\Gamma, G, X$ as in the previous subsection, in this subsection, we prove that $\Gamma$ fixes a point of $X$ in the case that $G$ contains a very proximal element with respect to some local field $(\mathbb{K}, \vert \cdot \vert)$.

\begin{lem}
    \label{lem: very proximal}
    Let $\K$ be a local field, $P = \mathbb{P}^2(\K)$ the projective plane over $\K$ and $H = PGL_3(\K)$. Suppose $\Gamma < G^{\oplus}$ is finitely generated and acts purely elliptically on $X$. If $G$ has a very proximal element $t$, then $\Gamma$ fixes a point of $X$. 
\end{lem}

First, we set some notation. If $t \in G$ is a very proximal element, let $p_+ \in P$ be the projectivization of the eigenspace of $t$ corresponding to its dominant eigenvalue and let $p_-$ be the projectivization of the minimal absolute value eigenspace of $t$. Denote $p$ the projectivization of the remaining one-dimensional eigenspace of $t$ (which has eigenvalue with absolute value strictly between the associated eigenvalues of $p_+$ and $p_-$). Denote $P_- = \Span(p_-, p)$ and $P_+ = \Span(p_+, p)$. 

Applying \cite[Corollary 2.87]{KapDrutu} to the proximal elements $t$ and $t^{-1}$, we obtain the following: 

\begin{lem}
    \label{very proximal NSD}
    Let $\K$ be a local field. Let $t$ be a very proximal element acting on $P \cong \mathbb{P}^d(\K)$ for some $d \in \N$.
    Then for all $\varepsilon > 0$, there exists $N \in \N$ such that for all $n \geq N$, for all $x \notin N_{\varepsilon}(P_-)$ we have $t^{n}(x) \in B_{\varepsilon}(p_+)$ and for all $y \notin N_{\varepsilon}(P_+)$, we have $t^{-n} (y) \in B_{\varepsilon}(p_-)$. 
\end{lem}

\begin{proof} [Proof of Lemma \ref{lem: very proximal}]
    Since $\Gamma$ acts purely elliptically on $X$, there exists a fixed point $z$ for $t$. Denote $E = P_+ \cup P_-$. 
    We first prove the following claim.

    \begin{clm}
    \label{lem: regularity lemma}
     Suppose that $r \in P \setminus E$ and $f \in G$ is such that $f(r) \notin E$. Then $f$ is biregular over $r$ with respect to $z$. 
\end{clm}

\begin{proof}[Proof of Claim \ref{lem: regularity lemma}]

    Let $f,r$ be as in the statement of the claim. Let $B$ denote the set of all $q \in P$ over which $f$ and $f^{-1}$ are singular. We consider the following cases. 

    \begin{enumerate}
        \item First, suppose $f(p_+) \notin P_-$ and $f^{-1}(p_-) \notin P_+$. We will show that there exists $n \in \N$ such that for all $i > 0$, $(t^n f)^i(r) \notin B$ and for all $i < 0$, $t^{-n}(t^n f)^i(r) \notin B$. If $f$ is singular over $r$, then it will follow that $t^n f$ has persistent fibre over $r$, contradicting that $t^nf$ fixes a point of $X$. We make the following choices of neighbourhoods separating points from the exceptional subspaces $P_-$ and $P_+$.
        
        \begin{itemize}
            \item Since $f(p_+) \notin P_-, f^{-1}(p_-) \notin P_+, r \notin P_+$ and $f(r) \notin P_-$, using that $\mathbb{P}^2$ is a regular topological space (since $\K$ is a locally compact field), we can choose $\ve_1 > 0$ such that $B_{\ve_1}(f(p_+)) \cap N_{\ve_1}(P_-) = \emptyset$, $B_{\ve_1}(f^{-1}(p_-)) \cap N_{\ve_1}(P_+) = \emptyset$, $B_{\ve_1}(f(r)) \cap N_{\ve_1}(P_-) = \emptyset$ and $B_{\ve_1}(r) \cap N_{\ve_1}(P_+) = \emptyset$.
            \item Using the continuity of $f$ and $f^{-1}$, there exists $0 < \ve < \ve_1$ such that $f(B_{\ve}(p_+)) \subset B_{\ve_1}(f(p_+))$ and $f^{-1}(B_{\ve}(p_-)) \subset B_{\ve_1}(f^{-1}(p_-))$. We arrange for $\ve > 0$ to be small enough such that $B_{\ve}(p_+) \cap (B \cup \{f^{-1}(p_+)\}) \subset \{p_+\}$ and $B_{\ve}(p_-) \cap (B \cup \{f(p_-)\}) \subset \{p_-\}$ (i.e.\ we separate $p_+$ from the finite set $B \cup \{f^{-1}(p_+)\}$ if $p_+$ is not in this set, and similarly with $p_-$).
            \item Choose $n \in \N$ as in Lemma \ref{very proximal NSD} for $t$ and $\ve$, and let $U_+ = B_{\ve}(p_+), U_- = B_{\ve}(p_-)$.
        \end{itemize}

        Refer to Figure \ref{fig:case 1 dyn} for the choices of neighbourhoods above and the dynamics of $t^{\pm n}$ and $f^{\pm 1}$. 

        \begin{figure}[H]
            \centering
            \includegraphics[width=0.8\linewidth]{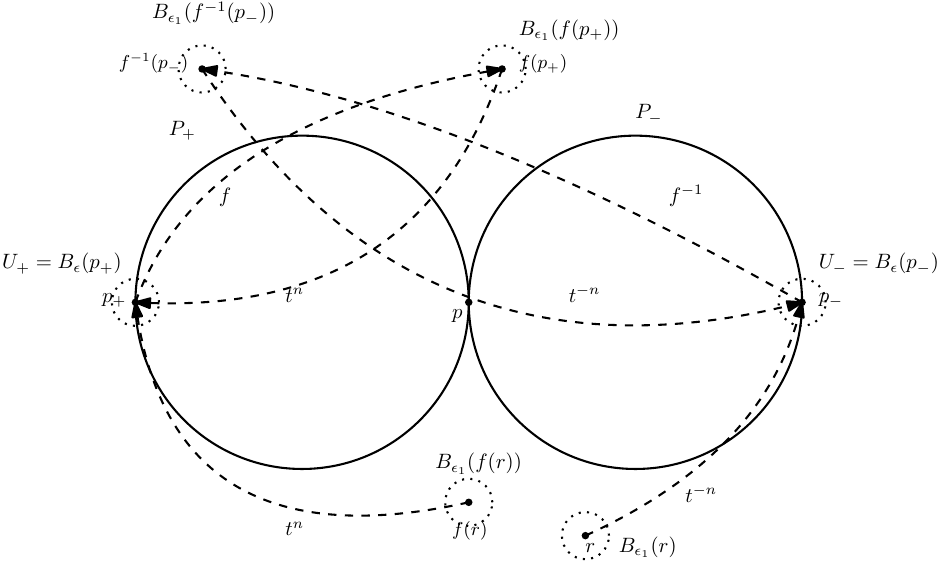}
            \caption{The dynamics of case 1.}
            \label{fig:case 1 dyn}
        \end{figure}

        We will show by induction on $m \geq 1$ that $(t^nf)^m (r) \in U_+ \setminus p_+$, which will imply that $(t^n f)^m (r) \notin B$ for all $m > 0$. 
        
        Since $f(r) \notin N_{\ve_1}(P_-)$ and since $\ve < \ve_1$, we have $f(r) \notin N_{\ve}(P_-)$ and so by Lemma \ref{very proximal NSD} and our choice of $n$, we have $t^nf(r) \in U_+ \setminus p_+$. Suppose $(t^nf)^m(r) \in U_+ \setminus p_+$ for some $m > 0$. Then $f(t^n f)^m(r) \in f(U_+) \subset B_{\ve_1}(f(p_+))$ which is disjoint from $N_{\ve}(P_-)$ by our choice of $\ve$, and hence $(t^n f)^{m+1}(r) = t^n f (t^n f)^m(r) \in U_+$. If  $(t^n f)^{m+1}(r) = p_+$, then $(t^n f)^m (r) = f^{-1}(p_+)$, which is a contradiction to $(t^nf)^m(r) \in U_+ \setminus p_+$ unless $f^{-1}(p_+) = p_+$, in which case $(t^n f)^{m+1}(r) = p_+$ implies $(t^n f)^{m}(r) = p_+$, a contradiction. Therefore, $(t^n f)^{m}(r) \in U_+ \setminus p_+$ for all $m > 0$. 

        Thus, $(t^nf)^m(r) \notin B$ for all $m > 0$. 

        Next, we show that $t^{-n}(t^nf)^m (r) \notin B$ for all $m <0$ by showing inductively that for each $m < 0$, we have $t^{-n} (t^n f)^m (r) \in U_- \setminus p_-$. 

        For $m = -1$, starting with $r \notin N_{\ve}(P_+)$, we have $t^{-n} (r) \in U_- \setminus p_-$, and thus, $f^{-1}t^{-n} (r) \in f^{-1}(U_-)$. Since $f^{-1}(U_-) \cap N_{\ve}(P_+) = \emptyset$, we have $t^{-n}(t^n f)^{-1}(r) = t^{-n}f^{-1}t^{-n}(r) \in U_-$. Furthermore, $t^{-n}(t^n f)^{-1}(r) \neq p_-$. Indeed, otherwise $t^{-n}(r) = f(p_-)$. If $f(p_-) \neq p_-$, then this contradicts that $U_- \setminus p_-$ does not contain $f(p_-)$, and if $f(p_-) = p_-$, then $r = p_-$, contradicting that $r \notin P_-$. 

        Now suppose that $t^{-n}(t^n f)^m (r) \in U_- \setminus p_-$ for some $m < 0$. We then have $f^{-1}t^{-n}(t^n f)^m (r) \in f^{-1}(U_-)$. Since $f^{-1}(U_-) \cap N_{\ve}(P_+) = \emptyset$, we have $t^{-n}(t^n f)^{m-1}(r) \in U_-$. If $t^{-n}(t^n f)^{-m-1}(r) = p_-$, then $(t^n f)^{m-1}(r) = p_-$, which implies $t^{-n} (t^n f)^m (r) = f(p_-)$. This is a contradiction unless $f(p_-) = p_-$, in which case $t^{-n}(t^n f)^{m-1}(r) = p_-$ implies $t^{-n}(t^n f)^{m}(r) = p_-$, a contradiction. Therefore, $t^{-n}(t^n f)^{m-1}(r) \in U_- \setminus p_-$. 

        Thus, $t^{-n}(t^nf)^m(r) \in U_- \setminus p_- \subset P \setminus B$ for all $m < 0$.

        If $f$ is singular over $r$, then $(t^nf)^m(r) \notin B$ for all $m > 0$ and $t^{-n}(t^nf)^m(r) \notin B$ for all $m < 0$ imply that $t^n f$ has persistent fibre over $r$. 

        Indeed, we will show that $(t^n f)^m$ is singular over $r$ for all $m > 0$ and $(t^n f)^m$ is biregular over $r$ for all $m < 0$. 
        \vs
        \underline{$(t^n f)^m$ is singular over $r$ for all $m > 0$}: 
        \vs
        For $m = 1$, if $t^n f$ is biregular over $r$, then since $t^{-n}$ is biregular over $t^n f(r)$, by \cite[Remark 4.4]{Algebraic}, we have that $f = t^{-n} t^n f$ is regular over $r$, a contradiction. Therefore, $t^n f$ is singular over $r$. 

        Suppose that for some $m > 0$, $(t^n f)^m$ is singular over $r$. If $(t^n f)^{m+1}$ is biregular over $r$, then since $t^{-n}$ is biregular over $(t^n f)^{m+1}(r)$, we have by \cite[Remark 4.4]{Algebraic} that $f(t^nf)^m = t^{-n} (t^n f)^{m+1}$ is biregular over $r$. Since $(t^nf)^m(r) \notin B$, we have that $f$ is biregular over $(t^nf)^m(r)$, and so by \cite[Remark 4.4]{Algebraic}, $f^{-1}$ is biregular over $f(t^nf)^m(r)$. Then again by \cite[Remark 4.4]{Algebraic}, we have that $(t^nf)^m = f^{-1}(f (t^nf)^m)$ is biregular over $r$, a contradiction. Therefore, $(t^nf)^m$ is singular over $r$ for all $m > 0$. 
\vs
        \underline{$(t^n f)^m$ is biregular over $r$ for all $m < 0$}:
\vs
        For $m = -1$, we have $t^{-n}(r) \notin B$, so $f^{-1}$ is biregular over $t^{-n}(r)$. By \cite[Remark 4.4]{Algebraic}, since $t^{-n}$ is biregular over $r$, we then have that $(t^nf)^{-1} = f^{-1}t^{-n}$ is biregular over $r$. 

        Suppose that $(t^nf)^m$ is biregular over $r$ for some $m < 0$. Since $t^{-n} (t^nf)^{m}(r) \notin B$, we have that $f^{-1}$ is biregular over $t^{-n} (t^nf)^{m-1}(r)$. Since $(t^nf)^m$ is biregular over $r$ and $t^{-n}$ is biregular over $(t^nf)^m(r)$, by \cite[Remark 4.4]{Algebraic}, we have that $t^{-n} (t^nf)^{m}$ is biregular over $r$. Lastly, since $t^{-n} (t^nf)^{m}$ is biregular over $r$ and $f^{-1}$ is biregular over $t^{-n} (t^nf)^{m}(r)$, one more application of \cite[Remark 4.4]{Algebraic} yields that $(t^n f)^{m-1} = f^{-1}t^{-n} (t^nf)^{m}(r)$ is biregular over $r$. Therefore, $(t^n f)^m$ is biregular over $r$ for each $m < 0$. 

        Thus, $t^nf$ has persistent fibre over $r$ (with $l = 1$), contradicting that $\Gamma$ acts purely elliptically. 

        Hence, we must have that $f$ is biregular over $r$. 

        \item Suppose exactly one of $f(p_+) \in P_-$ or $f^{-1}(p_-) \in P_+$ is true. Without loss of generality, we can assume that $f(p_+) \in P_-$ and $f^{-1}(p_-) \notin P_+$, since if $f^{-1}(p_-) \notin P_+$, then we can replace $t$ with $t^{-1}$ (which will swap $p_-$ with $p_+$ and $P_-$ with $P_+$), $f$ with $f^{-1}$ and $r$ with $f(r)$ and run the argument for $f(p_+) \in P_-$. 

        Then we must have $f(p) \notin P_-$, since otherwise if $f(p) \in P_-$, then $p_+, p \in f^{-1}(P_-)$, and so $P_+ = \Span(p_+, p) \subseteq f^{-1}(P_-)$, which implies $P_+ = f^{-1}(P_-)$, and hence that $f^{-1}(p_-) \in P_+$, a contradiction. 

        We will show that there exists $N,n \in \N$ such that for all $i \neq 0$, we have $(t^N f t^n f)^{i}(r) \notin B$. 

        First, by the argument in case 1, since $f^{-1}(p_-) \notin P_+$, there exists a neighbourhood $U_-$ of $p_-$ such that $U_- \setminus p_- \subset P \setminus B$ and $n_0 \in \N$ such that for all $m \geq n_0$ and all $i < 0$, we have $t^{-m}(t^m f)^i (r) \in U_- \setminus p_- \subset P\setminus B$. 

        Now, we choose separating neighbourhoods and powers of $t$ as follows (see Figure \ref{fig:case2}): 
        
        \begin{itemize}
            \item Since $f(p) \notin P_-$, there exists $\ve_1 > 0$ such that $B_{\ve_1}(f(p)) \cap N_{\ve_1}(P_-) = \emptyset$.
            \item Next, by continuity of $f$, there exists $\ve_2 > 0$ such that $f(B_{\ve_2}(p)) \subset B_{\ve_1}(f(p))$ and $f(p_+) \notin B_{\ve_2}(p_-)$ (which is possible since $f(p_+) \neq p_-$, since $f^{-1}(p_-) \notin P_+$). We can choose $\ve_2$ small enough such that $B_{\ve_2}(p) \cap (B \cup \{f^{-1}(p_+)\})  \subset \{p\}$. 
            \item Then by Lemma \ref{very proximal NSD} applied to $t$ acting on $P_-$, there exists $N \geq n_0$ such that $t^N f(p_+) \in B_{\ve_2}(p)$.
            \item By continuity of $t$, there exists $\ve_3 > 0$ such that $t^N (B_{\ve_3}(f(p_+))) \subset B_{\ve_2}(p)$. 
            \item Since $f(r) \notin P_-$, there exists $\ve_4 > 0$ such that $B_{\ve_4}(f(r)) \cap N_{\ve_4}(P_-) = \emptyset$. We can take $\ve_4 < \ve_1$ small enough such that $B_{\ve_4}(p_+) \cap (B \cup \{f^{-1}(p))\} \subset \{p_+\}$ and such that $f(B_{\ve_4}(p_+)) \subset B_{\ve_3}(f(p_+))$. 
            \item By Lemma \ref{very proximal NSD}, there exists  $n \geq n_0$ such that $t^n f(r) \in B_{\ve_4}(p_+)$ (choosing $n \geq n_0$ for $\ve_4$ in Lemma \ref{very proximal NSD}).  
        \end{itemize}

        \begin{figure}[H]
            \centering
            \includegraphics[width=0.8\linewidth]{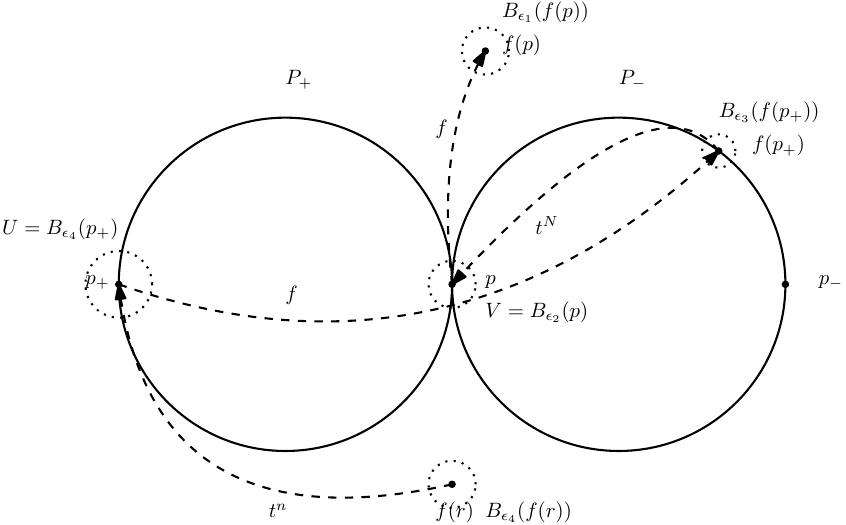}
            \caption{The dynamics of case 2.}
            \label{fig:case2}
        \end{figure}
        
        Let $U = B_{\ve_4}(p_+)$ and $V = B_{\ve_2}(p)$. We now show by induction on $i \in \N$ that for all $i > 0$, we have $(t^Nft^nf)^i(r) \in V \setminus p$ and for all $i \geq 0$, we have $t^nf (t^Nft^nf)^i(r) \in U \setminus p_+$. These imply that for all $i > 0$, we have $(t^Nft^nf)^i(r) \notin B$ by choice of $V$ and for all $i \geq 0$ we have $t^n f(t^Nft^nf)^i(r) \notin B$ by choice of $U$ and $V$.

        Since $B_{\ve_4}(f(r)) \cap N_{\ve_4}(P_-) = \emptyset$, by our choice of $n$ we have $t^n f(r) \in U$, and $t^n f(r) \neq p_+$ since $f(r) \notin P_+$. Then $ft^nf(r) \in f(U) \subset B_{\ve_3}(f(p_+))$, and hence $t^N ft^nf(r) \in V$ by choice of $N$. Furthermore, if $t^N f t^n f(r) = p$, then $t^n f(r) = f^{-1}(p)$. If $f^{-1}(p) \neq p_+$, then this is a contradiction to our choice of $U$. Otherwise, if $f^{-1}(p) = p_+$, then we have $t^n f(r) = p_+$, which implies $f(r) = p_+$, contradicting that $f(r) \notin P_+$. Therefore, $t^Nft^n f(r) \in V \setminus p$. 

        Suppose that $(t^N f t^n f)^i(r) \in V \setminus p$ and that $t^n f(t^N f t^n f)^{i-1}(r) \in U \setminus p_+$ for some $i > 0$. We first show that $t^n f(t^N f t^n f)^{i}(r) \in U \setminus p_+$. From $(t^N f t^n f)^i(r) \in V \setminus p$, we have $f(t^N f t^n f)^i(r) \in f(V) \subset B_{\ve_1}(f(p))$, which is disjoint from $N_{\ve_1}(P_-)$, and hence from $N_{\ve_4}(P_-)$, since $\ve_4 < \ve_1$. Then $t^n f(t^N f t^n f)^i(r) \in U$. If $t^n f(t^N f t^n f)^i(r) = p_+$, then $(t^N f t^n f)^i(r) = f^{-1}(p_+)$. If $f^{-1}(p_+) \neq p$, then this contradicts $(t^N f t^n f)^i(r) \in V \setminus p$. Otherwise, we obtain $(t^N f t^n f)^i(r) = p$, which implies $t^n f (t^Nft^nf)^{i-1}(r) = f^{-1}(p)$. If $f^{-1}(p) \neq p_+$, then this contradicts our choice of $U$, since $U$ separates $p_+$ from $f^{-1}(p)$ if they are not equal. If $f^{-1}(p) = p_+$, then $t^n f (t^Nft^nf)^{i-1}(r) = p_+$, which contradicts that $t^n f(t^N f t^n f)^{i-1}(r) \in U \setminus p_+$. Therefore, we must indeed have that $t^n f(t^N f t^n f)^{i}(r) \in U \setminus p_+$.  
        
        Next, we show $(t^N f t^n f)^{i+1}(r) \in V \setminus p$. From $t^n f(t^N f t^n f)^{i}(r) \in U$, we have $ft^n f(t^N f t^n f)^i(r) \in f(U) \subset B_{\ve_3}(f(p_+))$. By choice of $\ve_3$, we have $(t^N f t^n f)^{i+1}(r) = t^Nft^n f(t^N f t^n f)^i(r) \in V$. 

        To show that $(t^N f t^n f)^{i+1}(r) \neq p$, suppose otherwise. Then $t^n f (t^N f t^n f)^{i}(r) = f^{-1}(p)$. If $f^{-1}(p) \neq p_+$, then this is a contradiction if $f^{-1}(p) \neq p_+$ since by above $t^n f (t^N f t^n f)^{i}(r) \in U$. If $f^{-1}(p) = p_+$, then $t^n f (t^N f t^n f)^{i}(r) = f^{-1}(p) = p_+$, which contradicts that $t^n f(t^N f t^n f)^{i}(r) \in U \setminus p_+$. Therefore, $(t^N f t^n f)^{i+1}(r) \in V \setminus p$

        Therefore, we obtain for each $i > 0$ that $(t^N f t^n f)^i(r) \in V \setminus p$, and for each $i \geq 0$ that $t^nf(t^N f t^n f)^i(r) \in U \setminus p_+$. Hence, $(t^N f t^n f)^i(r) \notin B$ for each $i > 0$ and $t^nf(t^N f t^n f)^i(r) \notin B$ for each $i \geq 0$. 

        Next, from the proof of the first case, since we chose $N,n \geq n_0$ and since we have \newline $t^{-m}(t^m f)^{i}(r) \notin B$ for each $m \geq n_0$ and each $i < 0$, we have for each $i \leq 0$ that \newline $t^{-N}(t^N f t^n f)^{i}(r) \in U_- \setminus p_- \subset P \setminus B$, and that $t^{-n}f^{-1}t^{-N}(t^N f t^n f)^{i}(r) \in U_- \setminus p_- \subset P \setminus B$. 

        Supposing $f$ is singular over $r$ with respect to $z$, we will show that this then implies that $t^N f t^n f$ has persistent fibre over $r$ (with $l = 1$). 
\vs \vs
        \underline{$(t^N f t^n f)^i$ is singular over $r$ for all $i > 0$}: 
\vs
        For $i = 1$, if $t^N f t^n f$ was biregular over $r$, then by \cite[Remark 4.4]{Algebraic}, we would have $f t^n f$ biregular over $r$. But since $t^n f (r) \in U \setminus p_+ \subset P \setminus B$, we have that $f$ is biregular over $t^n f(r)$, and so $f^{-1}$ is biregular over $f t^n f(r)$ by \cite[Remark 4.4]{Algebraic}. This then yields that $t^n f = f^{-1} (ft^n f)$ is biregular over $r$, which yields that $f$ is biregular over $r$, a contradiction. Therefore, $t^N f t^n f$ is singular over $r$. 

        Suppose $(t^N f t^n f)^i$ is singular over $r$ for some $i > 0$, and that $(t^N f t^n f)^{i+1}$ is biregular over $r$. Then $ft^n f (t^N f t^n f)^i$ is biregular over $r$, which by \cite[Remark 4.4]{Algebraic} implies that $t^n f (t^N f t^n f)^i$ is biregular over $r$, since $t^n f (t^N f t^n f)^i \in U \setminus p_+ \subset P \setminus B$. This implies $f (t^N f t^n f)^i$ is biregular over $r$, which finally implies $(t^N f t^n f)^i$ is biregular over $r$ since $(t^N f t^n f)^i (r)\notin B$, and this is a contradiction. Therefore, $(t^N f t^n f)^{i+1}$ is singular over $r$.
\vs
        \underline{$(t^N f t^n f)^i$ is biregular over $r$ for all $i < 0$}: 
\vs
        Since $t^{-N}(r) \notin B$, we have $f^{-1}t^{-N}$ is biregular over $r$, and hence $t^{-n}f^{-1}t^{-N}$ is biregular over $r$. Then since $t^{-n}f^{-1}t^{-N}(r) \notin B$ (so that $f^{-1}$ is biregular over $t^{-n}f^{-1}t^{-N}(r)$) and since $t^{-n}f^{-1}t^{-N}$ is biregular over $r$, we have $(t^N f t^n f)^{-1} = f^{-1}t^{-n}f^{-1}t^{-N}$ is biregular over $r$ by \cite[Remark 4.4]{Algebraic}. 

        Suppose that $(t^N f t^n f)^i$ is biregular over $r$ for some $i < 0$. Then since $t^{-N}(t^N f t^n f)^i(r) \notin B$, and $t^{-N}(t^N f t^n f)^i$ is biregular over $r$, we have that $f^{-1}t^{-N}(t^N f t^n f)^i$ is biregular over $r$, which implies that $t^{-n}f^{-1}t^{-N}(t^N f t^n f)^i$ is biregular over $r$. Since $t^{-n}f^{-1}t^{-N}(t^N f t^n f)^i(r) \notin B$, we have that $f^{-1}$ is biregular over $t^{-n}f^{-1}t^{-N}(t^N f t^n f)^i(r)$ and hence that $(t^N f t^n f)^{i-1} = f^{-1}t^{-n}f^{-1}t^{-N}(t^N f t^n f)^i$ is biregular over $r$. 

        Therefore, $(t^N f t^n f)^i$ has persistent fibre over $r$, contradicting that $\Gamma$ acts purely elliptically on $X$. We conclude that $f$ is biregular over $r$.

    \item Lastly, suppose that $f(p_+) \in P_-$ and $f^{-1}(p_-) \in P_+$. 

    We consider the following subcases: 

    \begin{enumerate}
        \item First, suppose that $f(p_+) \neq p_-$ and $f(p_-) \neq p_+$. Then $P_- = f(P_+)$ and $P_+ = f(P_-)$, since $f(p_+) \neq p_-$ implies 
        
        $$P_- = \Span(p_-, f(p_+)) = f(\Span(f^{-1}(p_-), p_+)) = f(P_+)$$ 
        since $f^{-1}(p_-) \in P_+ \setminus p_+$. Similarly, $P_+ = f(P_-)$.
        
        Therefore, $f^2(P_{\pm}) = P_\pm$. It follows that $f^2(p) = p$, and hence the subgroup $S:= \langle t,f^2 \rangle$ fixes the point $p \in P$.
        By Lemma \ref{lem: virtual fixed point on P implies elliptic}, we have that $S$ fixes a point of $X$. This yields that $\langle t,f \rangle$ fixes a point $y$ of $X$, since $S$ has finite index in $\langle t, f \rangle$ (as $\langle t,f \rangle$ acts on the set of subspaces $\{P_-, P_+\}$, and $S$ is the kernel of this action). This implies that $f$ is biregular over $r$ with respect to $z$, since $z$ and $y$ agree on the infinite orbits of $t$ by Lemma \ref{Inf orb fix pt} (being common fixed points of $t$), so in particular $z_r = y_r$ for each $r \notin E$, and hence using that $y$ is a fixed point of $f$ and that $f(r) \notin E$, we have $z_r = y_r = y_{f(r)} = z_{f(r)}$, showing that $f$ is biregular over $r$ with respect to $z$.  

    \item Next, suppose that exactly one of $f(p_-) = p_+$ or $f(p_+) = p_-$ is true. We can assume that $f(p_-) = p_+$ and $f(p_+) \neq p_-$ since up to replacing $t$ by $t^{-1}$, we obtain the case $f(p_+) = p_-$ by running the argument for $f(p_-) = p_+$.
    
    Then we have that $f(p) \in P_-$ and $f^{-1}(p) \in P_+$, since $f(P_+) = f(\Span(p_+, f^{-1}(p_-)) = \Span(f(p_+), p_-) = P_-$. 
    
    We may assume that $f(p) \neq p$, or equivalently in our case, $f(p) \notin P_+$, since if $f(p) = p$ then $\langle t,f \rangle$ fixes $p$ and hence $\langle t,f \rangle < \Gamma$ fixes a point of $X$ by Lemma \ref{lem: virtual fixed point on P implies elliptic}, which implies that $f$ is biregular over $r$ by above. 

    We consider the following subcases.

    \begin{enumerate}

    \item First, suppose that $f(p_+) \neq p$ and $f^{-1}(p_-) \neq p$. Note then that $f(p_+) \in P_- \setminus p \subset P \setminus P_+$ and $f^{-1}(p_-) \in P_+ \setminus p \subset P \setminus P_-$. 
    
    We make the following choices of neighbourhoods:

    \begin{itemize}
        \item Choose $\ve_1 > 0$ such that: 
        
        \begin{itemize}
            \item $B_{\ve_1}(f(p_+)) \cap N_{\ve_1}(P_+) = \emptyset$, \item $B_{\ve_1}(f^{-1}(p_-)) \cap N_{\ve_1}(P_-) = \emptyset$, 
            \item $B_{\ve_1}(p_+) \cap N_{\ve_1}(P_-) = \emptyset$, 
            \item $B_{\ve_1}(p_+) \cap \{f^{-1}(p_-)\} \subset \{p_+\}$, \item $B_{\ve_1}(p_-) \cap N_{\ve_1}(P_+) = \emptyset$, 
            \item $B_{\ve_1}(p_-) \cap \{f(p_+)\} \subset \{p_-\}$, 
            \item $B_{\ve_1}(r) \cap N_{\ve_1}(E) = \emptyset$ and, 
            \item $B_{\ve_1}(f(r)) \cap N_{\ve_1}(E) = \emptyset$. 
        \end{itemize}

        We can furthermore decrease $\ve_1$ to ensure $B_{\ve_1}(f^{\pm 1}(p_\pm)) \cap B \subset \{f^{\pm 1}(p_\pm)\}$ and $B_{\ve_1}(p_\pm) \cap B \subset \{p_\pm\}$,
        \item Choose $\ve_2 > 0$ with $\ve_2 < \ve_1$ such that $f(B_{\ve_2}(p_+)) \subset B_{\ve_1}(f(p_+))$ and \newline $f^{-1}(B_{\ve_2}(p_-)) \subset B_{\ve_1}(f^{-1}(p_-))$.
        \item Choose $\ve > 0$ with $\ve < \ve_2$ such that $f(B_{\ve}(p_-)) \subset B_{\ve_2}(p_+)$ and $f^{-1}(B_{\ve}(p_+)) \subset B_{\ve_2}(p_-)$. 
    \end{itemize}

    Then choose $n \in \N$ for $\ve$ from Lemma \ref{very proximal NSD}. 
    
    Denote $U_+ = B_{\ve}(p_+)$, $V_+ = B_{\ve_1}(f(p_+))$, $V_- = B_{\ve_2}(p_-)$ and $U_- = B_{\ve}(p_-)$. We will show by induction that for all $i \geq 0$, we have: 

    \begin{itemize} 
        \item $t^n f(t^{-n}f t^n f)^i(r) \in U_+ \setminus p_+$,
        \item $(t^{-n}f t^n f)^{i+1}(r) \in U_- \setminus p_-$
    \end{itemize}

    For the base case $i = 0$:

    Since $f(r) \notin N_{\ve}(P_-)$, by Lemma \ref{very proximal NSD}, we have that $t^n f(r) \in U_+$. In addition, $t^n f(r) \neq p_+$ since $f(r) \notin P_+$. 

    Next, we have $ft^nf(r) \in f(U_+) \subset B_{\ve_1}(f(p_+)) = V_+$. If $ft^nf(r) = f(p_+)$, then $t^n f(r) = p_+$, which implies $f(r) = p_+$, a contradiction. Hence, $ft^nf(r) \in V_+ \setminus f(p_+)$. 
    
    Third, since $ft^nf(r) \in V_+$ which is disjoint from $N_{\ve}(P_+)$, we have $t^{-n} ft^nf(r) \in U_-$. If $t^{-n} ft^nf(r) = p_-$, then $t^n f(r) = f^{-1}(p_-) \in P_+$, which implies that $f(r) \in P_+$, a contradiction. Therefore, $t^{-n} ft^nf(r) \in U_- \setminus p_-$.


    This completes the base case. 

    Now suppose that for some $i \geq 0$, we have: 
    
     \begin{itemize} 
        \item $t^n f(t^{-n}f t^n f)^i(r) \in U_+ \setminus p_+$,
        \item $(t^{-n}f t^n f)^{i+1}(r) \in U_- \setminus p_-$
    \end{itemize}
    
    We will show that the same conditions hold for $i+1$:

     \begin{itemize} 
        \item $t^n f(t^{-n}f t^n f)^{i+1}(r) \in U_+ \setminus p_+$,
        \item $(t^{-n}f t^n f)^{i+2}(r) \in U_- \setminus p_-$
    \end{itemize}

    For the first item, since $(t^{-n}f t^n f)^{i+1}(r) \in U_- \setminus p_-$, we have that $f(t^{-n}f t^n f)^{i+1}(r) \in f(U_-) \subset B_{\ve_2}(p_+)$, which is disjoint from $N_{\ve}(P_-)$ and hence $t^nf(t^{-n}f t^n f)^{i+1}(r) \in U_+$. If $t^nf(t^{-n}f t^n f)^{i+1}(r) = p_+$, then $(t^{-n}f t^n f)^{i+1}(r) = f^{-1}(p_+) = p_-$, a contradiction. Therefore, $t^nf(t^{-n}f t^n f)^{i+1}(r) \in U_+ \setminus p_+$. 

    For the second item, since $t^n f(t^{-n}f t^n f)^{i+1}(r) \in U_+$, we have $ft^n f(t^{-n}f t^n f)^{i+1}(r) \in f(U_+) \subset B_{\ve_1}(f(p_+)) = V_+$. Then since $ft^n f(t^{-n}f t^n f)^{i+1}(r) \in V_+$, which is disjoint from $N_{\ve}(P_+)$, we have $(t^{-n}f t^n f)^{i+2}(r) = t^{-n}ft^n f(t^{-n}f t^n f)^{i+1}(r) \in U_-$. If $(t^{-n}f t^n f)^{i+2}(r) = p_-$, then $t^n f(t^{-n}f t^n f)^{i+1}(r) = f^{-1}(p_-)$. If $f^{-1}(p_-) \neq p_+$, then this contradicts our choice of $U_+$. Otherwise, we obtain $t^n f(t^{-n}f t^n f)^{i+1}(r) = p_+$, which is again a contradiction. Therefore, $(t^{-n}f t^n f)^{i+2}(r) \in U_- \setminus p_-$. 


    This completes our induction. Since each of $U_+, U_-$ separate, respectively, $p_+, p_-$ from $B$, we obtain from above that $t^n f(r) \notin B$, $ft^nf(r) \notin B$ and for each $i \geq 0$:

    \begin{itemize}
        \item $t^n f(t^{-n}f t^n f)^{i}(r) \notin B$,
        \item $(t^{-n}f t^n f)^{i+1}(r) \notin B$,
    \end{itemize}

    Analogously, we obtain $t^{-n} (r) \in U_- \setminus p_-$, $f^{-1} t^{-n}(r) \in V_- \setminus f^{-1}(p_-)$ and for every $i \leq 0$: 

    \begin{itemize} 
        \item $t^n(t^{-n}ft^n f)^{i} (r) \in U_+ \setminus p_+$
        \item $t^{-n}f^{-1}t^n(t^{-n}ft^n f)^{i} (r) \in U_- \setminus p_-$
    \end{itemize}

    which implies that for every $i \leq 0$: 

     \begin{itemize} 
        \item $t^n(t^{-n}ft^n f)^{i} (r) \notin B$,
        \item $t^{-n}f^{-1}t^n(t^{-n}ft^n f)^{i} (r) \notin B$,
    \end{itemize}

    Now suppose that $f$ is singular over $r$. We will show that $t^{-n} f t^n f$ has persistent fibre over $r$ for $l = 1$, i.e.\ that $(t^{-n} f t^n f)^i$ is singular over $r$ for all $i > 0$ and $(t^{-n} f t^n f)^{i}$ is biregular over $r$ for all $i < 0$. 

    First for positive $i$, suppose that $t^{-n} f t^n f$ was biregular over $r$. Then this implies $ft^n f$ is biregular over $r$ since $t$ fixes $z$. Since $t^n f (r) \notin B$, we have that $f^{-1}$ is biregular over $ft^n f (r)$, so by \cite[Remark 4.4]{Algebraic}, we obtain that $t^n f$ is biregular over $r$, which implies that $f$ is biregular over $r$, a contradiction. 

    Suppose that $(t^{-n} f t^n f)^i$ was singular over $r$ for some $i > 0$ but $(t^{-n} f t^n f)^{i+1}$ was biregular over $r$. Then $ft^n f(t^{-n} f t^n f)^i$ is biregular over $r$. Since $t^n f(t^{-n} f t^n f)^i(r) \notin B$, we obtain that $f^{-1}$ is biregular over $ft^n f(t^{-n} f t^n f)^i(r)$ and hence that \newline $t^n f(t^{-n} f t^n f)^i$ is biregular over $r$, which implies that $f(t^{-n} f t^n f)^i$ is biregular over $r$. Again, since $(t^{-n} f t^n f)^i(r) \notin B$, we obtain that $(t^{-n} f t^n f)^i$ is biregular over $r$, a contradiction. We conclude that $(t^{-n} f t^n f)^i$ is singular over $r$ for all $i > 0$. 

    For negative $i$, since $t^n$ is biregular over $r$ and $t^n r \notin B$, we obtain that $f^{-1} t^n $ is biregular over $r$. This implies $t^{-n} f^{-1} t^n$ is biregular over $r$. Since $t^{-n} f^{-1} t^n (r) \notin B$, we obtain that $f^{-1}t^{-n} f^{-1} t^n$ is biregular over $r$. 

    Suppose that $(t^{-n} f t^n f)^i$ was biregular over $r$ for some $i < 0$. Then $t^n (t^{-n} f t^n f)^i$ is biregular over $r$. Since $t^n (t^{-n} f t^n f)^i (r) \notin B$, we obtain that $f^{-1}t^n(t^{-n} f t^n f)^i$ is biregular over $r$, which yields that $t^{-n}f^{-1}t^n(t^{-n} f t^n f)^i$ is biregular over $r$. Lastly, since $t^{-n}f^{-1}t^n(t^{-n} f t^n f)^i (r) \notin B$, we obtain that $(t^{-n} f t^n f)^{i-1} = f^{-1}t^{-n}f^{-1}t^n(t^{-n} f t^n f)^i$ is biregular over $r$. We conclude that $(t^{-n} f t^n f)^i$ is biregular over $r$ for all $i < 0$. 

    Therefore, $t^{-n} f t^n f$ has persistent fibre over $r$, contradicting that it fixes a point of $X$. We conclude that $f$ is biregular over $r$.

    \item Now consider the case that $f(p_+) = p$. Since $f(p) \neq p$, we obtain that $f^2(p_+) = f(p) \in P_- \setminus p$. Thus, $f^2(p_+) \notin P_+$. This also implies that $f^{-2}(p_-) \notin P_-$. Indeed, note that $f^{-2}(p_-) \in f^{-1}(P_+) = \Span(f^{-1}(p_+), f^{-1}(p)) = \Span(p_-, p_+)$, so if $f^{-2}(p_-) \in P_-$, then we would obtain that $f^{-2}(p_-) \in P_- \cap \Span(p_+, p_-) = \{p_-\}$, so $f^{-2}(p_-) = p_-$. But then we would have $p_- = f^{2}(p_-) = f(p_+) = p$, a contradiction. 

    We will argue similarly to above that there exists $N,n \in \N$ such that $t^{-n}f^2 t^N f$ has persistent fibre over $r$ if $f$ is singular over $r$. 

    Let $B_2$ denote the finite set of all $q \in P$ over which $f, f^{-1}, f^2$ or $f^{-2}$ are singular. 


    \begin{itemize}
        \item Choose $\ve_1 > 0$ such that $B_{\ve_1}(f(r)) \cap N_{\ve_1}(P_-) = \emptyset$, $B_{\ve_1}(f^2(p_+)) \cap N_{\ve_1}(P_+) = \emptyset$, $B_{\ve_1}(p_+) \cap N_{\ve_1}(P_-) = \emptyset$, $B_{\ve_1}(p_+) \cap \{f^{-2}(p_-)\} \subset \{p_+\}$, $B_{\ve_1}(p_+) \cap B_2 \subset \{p_+\}$, $B_{\ve_1}(p_-) \cap B_2 \subset \{p_-\}$ and $B_{\ve_1}(f^2(p_+)) \cap B_2 \subset \{f^2(p_+)\}$.
        \item Choose $\ve_2 < \ve_1$ such that $f^2(B_{\ve_2}(p_+)) \subset B_{\ve_1}(f^2(p_+))$.
        \item Choose $\ve_3 < \ve_2$ such that $f(B_{\ve_3}(p_-)) \subset B_{\ve_2}(p_+)$.
        \item Choose $n_0$ for $\ve_3$ from Lemma \ref{very proximal NSD} for $t$ acting on $P$. 
    \end{itemize}

    Denote $U_+ = B_{\ve_2}(p_+)$, $V = B_{\ve_1}(f^2(p_+))$, $U_- = B_{\ve_3}(p_-)$. We will show by induction that for all $i \geq 0$ and for any $m \in \N$ with $m,n \geq n_0$, we have: 

    \begin{itemize}
        \item $t^n f(t^{-m}f^2 t^n f)^i(r) \in U_+ \setminus p_+$
        \item $(t^{-m}f^2 t^n f)^{i+1}(r) \in U_- \setminus p_-$
    \end{itemize}

    Fix $m,n \geq n_0$. For the base case $i=0$, since $f(r) \notin N_{\ve_1}(P_-)$ and $\ve_3 < \ve_1$, by Lemma \ref{very proximal NSD}, we have that $t^n f(r) \in U_+$. In addition, $t^n f(r) \neq p_+$ because $f(r) \notin P_+$. 
    Next, we have that $f^2t^n f(r) \in f^2(U_+) \subset B_{\ve_1}(f^2(p_+)) = V$.
    Third, since $f^2 t^n f(r) \in V$, which is disjoint from $N_{\ve_3}(P_+)$, we have that $t^{-m}f^2 t^n f(r) \in U_-$. If $t^{-m}f^2 t^n f(r) = p_-$, then $t^n f(r) = f^{-2}(p_-)$. But since $t^n f(r) \in U_+$, we have that $f^{-2}(p_-) = p_+$ since $U_+ \subset B_{\ve_1}(p_+)$ which separates $p_+$ from $f^{-2}(p_-)$ if they are not equal. This implies that $f(p) = p_-$, and so $f$ cyclically permutes the set $\{p_+, p, p_-\}$ in the order listed, so that $f^3(p) = p$. We then have that $\langle t, f^3 \rangle$ fixes $p$, hence fixes a point of $X$ by Lemma \ref{lem: virtual fixed point on P implies elliptic}, which implies that $\langle t, f \rangle$ fixes a point of $X$ since $\langle t, f \rangle$ acts on $\{p_+, p, p_-\}$ with kernel $\langle t, f^3 \rangle$. As argued before, this yields that $f$ is biregular over $r$. Therefore, $t^{-m}f^2 t^n f(r) = p_-$ yields a contradiction. 
    This completes the base case. 

    For the induction step, suppose that the following hold for some $i \geq 0$. 

    \begin{itemize}
        \item $t^n f(t^{-m}f^2 t^n f)^i(r) \in U_+ \setminus p_+$
        \item $(t^{-m}f^2 t^n f)^{i+1}(r) \in U_- \setminus p_-$
    \end{itemize}

    We will prove that the same statements hold for $i+1$. 

    First, since $f(t^{-m}f^2 t^n f)^{i+1}(r) \in U_+ \setminus p_+$, and since $U_+$ is disjoint from $N_{\ve_3}(P_-)$, we have $t^nf(t^{-m}f^2 t^n f)^{i+1}(r) \in U_+$. If $t^nf(t^{-m}f^2 t^n f)^{i+1}(r) = p_+$, then \newline $f(t^{-m}f^2 t^n f)^{i+1}(r) = p_+$, which implies $(t^{-m}f^2 t^n f)^{i+1}(r) = f^{-2}(p_+)$. If $f^{-2}(p_+) \neq p_-$, then this contradicts our choice of $U_-$. Otherwise, if $f^{-2}(p_+) = p_-$, then $f(p_-) = p_+ = f^2(p_-)$, hence $f(p_-) = p_-$, which implies that $f$ is biregular over $r$ as argued above, since $\langle t, f \rangle$ fixes a point of $X$. 

    Next, $f^2t^nf(t^{-m}f^2 t^n f)^{i+1}(r) \in f^2(U_+) \subset V$, and so 
    $$(t^{-m}f^2 t^n f)^{i+2}(r) = t^{-m}f^2t^nf(t^{-m}f^2 t^n f)^{i+1}(r) \in t^{-m}(V)$$ and $t^{-m}(V) \subset U_-$ since $V \cap N_{\ve_1}(P_-) = \emptyset$ and $\ve_3 < \ve_1$. Therefore, $(t^{-m}f^2 t^n f)^{i+2}(r) \in U_-$. If $(t^{-m}f^2 t^n f)^{i+2}(r) = p_-$, then $f^2t^nf(t^{-m}f^2 t^n f)^{i+1}(r) = p_-$, which implies $$t^nf(t^{-m}f^2 t^n f)^{i+1}(r) = f^{-2}(p_-)$$ However, since $t^nf(t^{-m}f^2 t^n f)^{i+1}(r) \in U_+$, we obtain that $t^nf(t^{-m}f^2 t^n f)^{i+1}(r) = p_+$, which contradicts that $t^nf(t^{-m}f^2 t^n f)^{i+1}(r) \in U_+ \setminus p_+$ from above. Therefore, $(t^{-m}f^2 t^n f)^{i+2}(r) \in U_- \setminus p_-$.

    This completes the induction, and we conclude that 

    \begin{itemize}
        \item $t^n f(t^{-m}f^2 t^n f)^i(r) \in U_+ \setminus p_+$
        \item $(t^{-m}f^2 t^n f)^{i+1}(r) \in U_- \setminus p_-$
    \end{itemize}

    for all $i \geq 0$. Since $U_+ \setminus p_+, V \setminus f^2(p_+), U_- \setminus p_-$ are all disjoint from $B_2$ by choice of $\ve_1$, we obtain in particular that for all $i \geq 0$,

    \begin{itemize}
        \item $t^n f(t^{-m}f^2 t^n f)^i(r) \notin B_2$
        \item $(t^{-m}f^2 t^n f)^{i+1}(r) \in \notin B_2$
    \end{itemize}

    Now suppose that $f$ is singular over $r$. We will then show that $(t^{-m}f^2 t^n f)^i$ is singular over $r$ for all $i > 0$. We proceed by induction on $i$. 

    For the base case $i = 1$, suppose that $t^{-m}f^2 t^n f$ was biregular over $r$. Since $t^m$ is biregular over $t^{-m}f^2 t^n f (r)$, by \cite[Remark 4.4]{Algebraic}, we obtain that $f^2 t^n f = t^m (t^{-m}f^2 t^n f)$ is biregular over $r$. Next, since $t^n f (r) \notin B_2$, we have that $f^2$ is biregular over $t^n f (r)$ and so $f^{-2}$ is biregular over $f^2 t^n f (r)$ and hence by \cite[Remark 4.4]{Algebraic}, we have that $t^n f = f^{-2}(f^2 t^n f)$ is biregular over $r$. Lastly, since $t^{-n}$ is biregular over $t^n f(r)$, we obtain that $f$ is biregular over $r$, a contradiction. We conclude that  $t^{-m}f^2 t^n f$ is singular over $r$. 

    Now suppose that $(t^{-m}f^2 t^n f)^i$ were singular over $r$ for some $i > 0$, and suppose for contradiction that $(t^{-m}f^2 t^n f)^{i+1}$ was biregular over $r$. As above, by \cite[Remark 4.4]{Algebraic} we then obtain that $f^2 t^n f (t^{-m}f^2 t^n f)^i$ is biregular over $r$. Since $t^n f (t^{-m}f^2 t^n f)^i (r) \notin B_2$, we have that $f^{-2}$ is biregular over $f^2 t^n f (t^{-m}f^2 t^n f)^i (r)$ and hence $t^n f (t^{-m}f^2 t^n f)^i$ is biregular over $r$. This then implies by \cite[Remark 4.4]{Algebraic} that $f (t^{-m}f^2 t^n f)^i$ is biregular over $r$. Lastly, since $(t^{-m}f^2 t^n f)^i(r) \notin B_2$, we obtain that $f^{-1}$ is biregular over $f(t^{-m}f^2 t^n f)^i(r) \notin B_2$ and hence that $(t^{-m}f^2 t^n f)^i$ is biregular over $r$, contradicting our induction hypothesis. 

    We conclude that if $f$ is singular over $r$, then $(t^{-m}f^2 t^n f)^i$ is singular over $r$ for all $i > 0$. 

    For $i \leq 0$, keeping the same notation as above, choose $N \geq n_0$ $\ve > 0$ and Lemma \ref{very proximal NSD} applied to $t$ acting on $P_+$. Choose $\ve_3 > 0$ such that $t^{-N}(B_{\ve_3}(f^{-2}(p_+))) \subset B_{\ve}(p)$.  

    By above, there exists $n_0 \in \N$ such that for all $m,n \geq n_0$, and all $i > 0$, we have $(t^{-m}f^2 t^n f)^i(r) \notin B_2$. We will show that there exists $N,n \geq n_0$ such that for all $i < 0$, we have $(t^{-n}f^2 t^N f)^i (r) \notin B_2$.

    \begin{itemize}
            \item Since $p_+ \notin P_-$, there exists $\ve_1 > 0$ such that $B_{\ve_1}(p_+) \cap N_{\ve_1}(P_-) = \emptyset$ and $B_{\ve_1}(r) \cap N_{\ve_1}(P_-) = \emptyset$.
            \item Next, by continuity of $f^{-1}$, there exists $\ve_2 > 0$ such that \newline $f^{-1}(B_{\ve_2}(p)) \subset B_{\ve_1}(f^{-1}(p)) = B_{\ve_1}(p_+)$ and $f^{-2}(p_+) \notin B_{\ve_2}(p)$ (which is possible since we may assume that $f^{-2}(p_+) \neq p$, since otherwise $f^{-1}(p_-)=f^{-2}(p_+) = p$, which implies that $f$ permutes the set $\{p_+, p, p_-\}$, which as argued above yields that $f$ is biregular over $r$). Since $f(p) \neq p$, we can also choose $\ve_2$ small enough such that $f(p) \notin B_{\ve_2}(p)$. Lastly, we can choose $\ve_2$ small enough such that $B_{\ve_2}(p) \cap B  \subset \{p\}$. 
            \item Then by Lemma \ref{very proximal NSD} applied to $t$ acting on $P_+$, there exists $N \geq n_0$ such that $t^{-N} f^{-2}(p_+) \in B_{\ve_2}(p)$.
            \item By continuity of $t$, there exists $\ve_3 > 0$ such that $t^{-N} (B_{\ve_3}(f^{-2}(p_+))) \subset B_{\ve_2}(p)$. 
            \item By continuity of $f^{-2}$, there exists $\ve_4 < \ve_1$ such that \newline $f^{-2}(B_{\ve_4}(p_+)) \subset B_{\ve_3}(f^{-2}(p_+))$. We can take $\ve_4$ small enough such that $B_{\ve_4}(p_+) \cap (\{f^2(p)\} \cup B_2)\subset \{p_+\}$. 
            \item By Lemma \ref{very proximal NSD}, there exists  $n \geq n_0$ such that $t^n f(r) \in B_{\ve_4}(p_+)$.
            \item By Lemma \ref{very proximal NSD} applied to $t$ acting on $P$, there exists $n \geq n_0$ such that $t^n(P \setminus N_{\ve_4}(P_-)) \subset B_{\ve_4}(p_+)$. 
        \end{itemize}
        
        Let $U = B_{\ve_4}(p_+)$ and $V = B_{\ve_2}(p)$. We now show by induction on $i \in \N$ that for all $i \leq 0$, we have $t^n (t^{-n}f^2t^Nf)^{i}(r) \in U \setminus p_+$ and $t^{-N}f^{-2}t^n (t^{-n}f^2t^Nf)^{i}(r) \in V \setminus p$.

        For the base case $i=0$: 

        \begin{itemize}
            \item Since $\ve_4 < \ve_1$ and $r \notin N_{\ve_1}(P_-)$, we have $t^n (r) \in U$. We have that $t^n (r) \neq p_+$ since $r \neq p_+$.
            \item Since $t^n (r) \in U$, we have $f^{-2}t^n (r) \in f^{-2}(U) \subset B_{\ve_3}(f^{-2}(p_+))$. By choice of $N$, we then have $t^{-N}f^{-2}t^n (r) \in t^{-N}B_{\ve_3}(f^{-2}(p_+)) \subset B_{\ve_2}(p) = V$. If $t^{-N}f^{-2}t^n (r) = p$, then $t^n (r) = f^2(p)$. Since we may assume $f^2(p) \neq p_+$, this contradicts our choice of $\ve_4$. Therefore, $t^{-N}f^{-2}t^n (r) \in V \setminus p$.
        \end{itemize}

        This completes our base case. 

        Now suppose that for some $i \geq 0$ we have: 

        \begin{itemize}
            \item $t^n (t^{-n}f^2t^Nf)^{i}(r) \in U \setminus p_+$ and
            \item $t^{-N}f^{-2}t^n (t^{-n}f^2t^Nf)^{i}(r) \in V \setminus p$.
        \end{itemize}

        We will show the same conditions hold for $i - 1$. 

        \begin{itemize}
            \item Since $t^{-N}f^{-2}t^n (t^{-n}f^2t^Nf)^{i}(r) \in V$, we have \newline $f^{-1}t^{-N}f^{-2}t^n (t^{-n}f^2t^Nf)^{i}(r) \in f^{-1}(V) \subset B_{\ve_1}(p_+)$. Since $\ve_4 < \ve_1$ and since $B_{\ve_1}(p_+) \cap N_{\ve_1}(P_-)=\emptyset$, we have $$t^n ((t^{-n}f^2t^Nf)^{i-1})(r) = t^n f^{-1}t^{-N}f^{-2}t^n (t^{-n}f^2t^Nf)^{i}(r) \in U$$ If $t^n ((t^{-n}f^2t^Nf)^{i-1})(r) = p_+$, then we would obtain $t^{-N}f^{-2}t^n (t^{-n}f^2t^Nf)^{i}(r) = f(p) \neq p$, which contradicts our choice of $\ve_2$. Therefore, $t^n (t^{-n}f^2t^Nf)^{i-1}(r) \in U \setminus p_+$.
            \item Next, since $t^n (t^{-n}f^2t^Nf)^{i-1}(r) \in U$, we have that \newline $f^{-2}t^n (t^{-n}f^2t^Nf)^{i-1}(r) \in f^{-2}(U) \subset B_{\ve_3}(f^{-2}(p_+))$. Then by choice of $N$, we have \newline $t^{-N}f^{-2}t^n (t^{-n}f^2t^Nf)^{i-1}(r) \in t^{-N} (B_{\ve_3}(f^{-2}(p_+)) \subset V$. If \newline $t^{-N}f^{-2}t^n (t^{-n}f^2t^Nf)^{i-1}(r) = p$, then $$t^n (t^{-n}f^2t^Nf)^{i-1}(r) = f^2(p)$$ which is a contradiction since $t^n (t^{-n}f^2t^Nf)^{i-1}(r) \in U$ and $f^2(p) \notin U$.   Therefore, $t^{-N}f^{-2}t^n (t^{-n}f^2t^Nf)^{i-1}(r) \in V \setminus p$.
        \end{itemize}

        We will now show by induction on $i$ that for every $i < 0$, we have that $(t^{-n}f^2 t^N f)^i$ is biregular over $r$. 

        For the base case $i = -1$, we have $(t^{-n}f^2 t^N f)^{-1} = f^{-1}t^{-N}f^{-2}t^n$. We have that $t^n$ is biregular over $r$ and $t^n (r) \in U \setminus p_+$ which is disjoint from $B_2$, hence $f^{-2}$ is biregular over $t^n (r)$ and so by \cite[Remark 4.4]{Algebraic}, we have that $f^{-2}t^n $ is biregular over $r$. Next, since $t^{-N}$ is biregular over $f^{-2}t^n(r)$, we obtain by \cite[Remark 4.4]{Algebraic} that $t^{-N}f^{-2}t^n$ is biregular over $r$. Lastly, since $t^{-N}f^{-2}t^n(r) \in V \setminus p$ which is disjoint from $B_2$, we obtain that $f^{-1}$ is biregular over $t^{-N}f^{-2}t^n(r)$ and so $f^{-1}t^{-N}f^{-2}t^n$ is biregular over $r$. 

        For the induction step, suppose that $(t^{-n}f^2 t^N f)^i$ is biregular over $r$ for some $i < 0$. Since $t^n$ is biregular over $(t^{-n}f^2 t^N f)^i(r)$, we obtain that $t^n (t^{-n}f^2 t^N f)^i$ is biregular over $r$. By above, we have that $t^n (t^{-n}f^2 t^N f)^i (r) \in U \setminus p_+$ which is disjoint from $B_2$ and hence $f^{-2}$ is biregular over $t^n (t^{-n}f^2 t^N f)^i (r)$, which yields that $f^{-2}t^n (t^{-n}f^2 t^N f)^i$ is biregular over $r$. We then have that $t^{-N}f^{-2}t^n (t^{-n}f^2 t^N f)^i$ is biregular over $r$. Lastly, since $t^{-N}f^{-2}t^n (t^{-n}f^2 t^N f)^i(r) \in V \setminus p$, which is disjoint from $B_2$, we obtain that $f^{-1}$ is biregular over $t^{-N}f^{-2}t^n (t^{-n}f^2 t^N f)^i(r)$ and hence that $(t^{-n}f^2 t^N f)^{i-1} = f^{-1}t^{-N}f^{-2}t^n (t^{-n}f^2 t^N f)^i$ is biregular over $r$. 

        Therefore, for every $i < 0$, we have that $(t^{-n}f^2 t^N f)^i$ is biregular over $r$. 

       Combining this with the result above for $i > 0$, we conclude that if $f$ is singular over $r$, then $(t^{-n}f^2 t^N f)^i$ has persistent fibre over $r$ (for $l = 1)$, contradicting that $\Gamma$ acts purely elliptically on $X$.

    The case that $f^{-1}(p_-) = p$ is analogous and we omit the details (or note that we can obtain the case $f^{-1}(p_-) = p$ by replacing $t$ with $t^{-1}$, $f$ with $f^{-1}$ and $r$ with $f(r)$ and running the above argument for $f(p_+)=p$). 
    
    \end{enumerate}

        \item Lastly, suppose that $f(p_+) = p_-$ and $f(p_-) = p_+$. Then $f^2(p_{\pm}) = p_{\pm}$ and so $\langle t,f^2 \rangle$ fixes the points $p_{\pm}$ of $P$, which implies that $\langle t,f^2 \rangle$ is an elliptic subgroup of $\Gamma$ by Lemma \ref{lem: virtual fixed point on P implies elliptic}. Since $\langle t,f^2 \rangle$ has finite index in $\langle t, f \rangle$ (being the kernel of the action of $\langle t, f \rangle$ on $\{p_-, p_+\}$), we obtain that $\langle t, f \rangle$ is elliptic, which implies that $f$ is biregular over $r$ as argued above. 
        
        \end{enumerate}
        \end{enumerate}
        \end{proof}

        We will now adjust the coordinates of the fixed point $z$ of our very proximal element $t$ on $E := P_+ \cup P_-$ to obtain a fixed point of $\Gamma$ on $X$. 

             Note that if $G$ preserves $E$, then denoting $G^2 = \langle g^2 : g \in G \rangle$, we have that $G^2 (P_\pm)  = P_\pm$. 

             In this case, in the basis $(p,p_+,p_-)$, we have $G^2 < \begin{bmatrix}
                    1 & * & * \\
                    0 & * & 0 \\
                    0 & 0 & *
                \end{bmatrix}$, so that $G^2(p) = p$. By Lemma \ref{lem: virtual fixed point on P implies elliptic}, we obtain that $\Gamma$ fixes a point of $X$ since $G^2$ has finite index in $G$ (indeed, $G^2 \triangleleft G$ and $G/G^2$ is a finitely generated torsion abelian group, hence is finite). 

            Also, if some $q \in E$ has a finite orbit $Gq$, then by \cite[Remark 4.2]{Algebraic}, we can adjust the coordinates of $z$ on $Gq$ to obtain a point $z'$ such that each $f \in G$ is biregular over each $r \in Gq$ with respect to $z'$. 

            Therefore, we can assume $GE \nsubseteq E$ and that each $q \in P$ has infinite orbit $Gq$. 

            First, suppose that $Gq \subset E$. We may assume without loss of generality that $q \in P_+$ and $Gq \cap P_+$ is infinite. Let $g_0 \in G$ be such that $g_0 q \in P_+ \setminus q$. Then we have that $P_+ = \mathrm{Span}(q, g_0q)$. We will show that for all $g \in G$, either $gq, gg_0q \in P_+$ or $gq, gg_0q \in P_-$. Indeed, if not, say $gq \in P_+ \setminus P_-$ and $gg_0 q \in P_- \setminus P_+$, then $\mathrm{Span}(gq, gg_0q) \neq P_+, P_-$, and hence $\mathrm{Span}(gq, gg_0q)$ intersects $E$ in only the points $gq$ and $gg_0q$, and so since $Gq \subset E$, we have that $\mathrm{Span}(gq, gg_0q) \cap Gq$ is finite. But $\vert \mathrm{Span}(gq, gg_0q) \cap Gq \vert = \vert g \mathrm{Span}(q,g_0q) \cap Gq \vert = \vert \mathrm{Span}(q,g_0q) \cap Gq\vert = \vert P_+ \cap Gq\vert$, which is infinite, so we have a contradiction. Therefore, we must have for all $g \in G$, either $gq, gg_0q \in P_+$ or $gq, gg_0q \in P_-$. 
             
             If $gq, gg_0q \in P_+$ for all $g \in G$, then $GP_+ = P_+$ and so by Theorem \ref{thm: Algebraic main theorem} or \cite[Lemma 4.9]{Algebraic}, we can adjust the coordinates of $z$ on $P_+ \cong \mathbb{P}^1$ to obtain a fixed point of $G$ on $P_+$, so that each $f \in G$ becomes biregular over $q$ with respect to this new point. In the other case, if there exists $g \in G$ with $gq, gg_0q \in P_-$, then $g P_+ = P_-$ and $Gq \cap P_-$ is infinite. Thus, repeating the above argument for $P_-$, we obtain that for each $g \in G$, either $g$ preserves both $P_-$ and $P_+$ or $g$ interchanges $P_-$ and $P_+$, and so $G^2 P_\pm = P_\pm$. This reduces us to the first case considered and we obtain that $\Gamma$ fixes a point of $X$ by Lemma \ref{lem: virtual fixed point on P implies elliptic}. Therefore, in the case that $Gq \subset E$ and is infinite, we can obtain a point $z' \in X$ such that each $f \in G$ is biregular over $q$ with respect to $z'$. We can thus assume that $Gq \nsubseteq E$.
             
             Let $q \in E$ and suppose that $Gq \nsubseteq E$. Then there exists $r \notin E$ and $f \in G$ such that $f(r) = q$. If $f' \in G$ and $r' \notin E$ are such that $f'(r') = q$, then $f^{-1}f'(r') = r$ and so by Lemma \ref{lem: regularity lemma}, we have that $f^{-1}f'$ is biregular over $r'$ with respect to $z$, i.e.\ $f^{-1}f'(z)_{r'} = z_r$, so that $f'(z)_q = f(z)_q$. We then replace $z_q$ with $f(z)_q$, which as we have just observed, does not depend on $f \in G$ and $r \notin E$ such that $f(r) = q$. Define a new point $z' \in X$ by setting $z'_q = f(z)_q$ if $q \in E$ is such that $Gq \nsubseteq E$ and each $f \in G$ such that $f(r) = q$ for some $r \notin E$ and $z'_q = z_q$ otherwise. By construction of $z'$, we have that $f$ is biregular over $r$ with respect to $z'$ and hence $f^{-1}$ is biregular over $f(r) = q$ with respect to $z'$. Furthermore, $z'$ is still a fixed point of $t$. Indeed, if $q \in E$ is such that $Gq \nsubseteq E$, let $f \in G$ and $r \notin E$ be such that $f(r) = q$, and put $f' = tf$, so that $t = f'f^{-1}$. Then $f(r) = q$ and $f'(r) = t(q) \in E$, so by above $f^{-1}$ is biregular over $q$ with respect to $z'$ and $f'$ is biregular over $r = f^{-1}(q)$ with respect to $z'$. By \cite[Remark 4.4]{Algebraic}, it follows that $t$ is biregular over $q$ with respect to $z'$. Otherwise, if $q \notin E$ or $Gq \subset E$, then $z'_q= z_q$ and $z'_{t(q)} = z_{t(q)}$, and since $t$ is biregular over $q$ with respect to $z$ (as $t(z) = z$), we obtain that $t$ is biregular over $q$ with respect to $z'$. Therefore, $z'$ is still fixed by $t$.

            We next show biregularity with respect to $z'$ of $f \in G$ over $q \in E$ such that $f(q) \in E$. By above, we can assume that there is $f' \in G$ and $r \notin E$ such that $f'(r) = q$ (if not, then $Gq \subset E$, and we have already shown that $\Gamma$ acts elliptically in this case). Putting $f'' = ff'$, we have $f''(r) = f(q) \in E$, and so by Claim \ref{lem: regularity lemma}, $f'',f'$ are both biregular over $r$ with respect to $z'$, hence $f$ is biregular over $q$ with respect to $z'$ (since $f = f''(f')^{-1}$ and $(f')^{-1}$ is biregular over $f'(r) = q$ and $f''$ is biregular over $(f')^{-1}(q) = r$, so that $f$ is biregular over $q$ by \cite[Remark 4.4]{Algebraic}).

            Now suppose that $q \notin E$. If $f \in G$ is such that $f(q) \notin E$, then $f$ is biregular over $q$ with respect to $z'$ by Claim \ref{lem: regularity lemma}. If $f(q) \in E$, then by construction of $z'$ above, we have that $f$ is biregular over $q$ with respect to $z'$ (replacing $q$ with $r$ and $f(q)$ with $q$ in the notation of the paragraph where the coordinates $z'_q$ are defined for $q \in E$ such that $Gq \nsubseteq E$). 

            Thus, adjusting the coordinates of the fixed point $z$ of $t$ as above, we obtain a new point $z'$ such that each $f \in G$ is biregular over all $q \in P$ with respect to $z'$, hence $z'$ is a fixed point of $\Gamma$. 
    
\end{proof}

\subsection{The conclusion of the proof of Theorem \ref{thm: main thm}}

We are now ready to conclude the proof of Theorem \ref{thm: main thm}. We keep the same notation as used throughout the previous subsections. First, a couple of lemmas.

\begin{lem}
    \label{lem: No proximal implies virtually nilpotent}
    Suppose that $G$ does not have a proximal element with respect to any local extension of a finitely generated subfield of $k$. Then choosing representatives in $GL_3(k)$ for each $g \in G$ with $\det g = 1$, we have that every $g \in G$ has all eigenvalues roots of unity. In particular, $G$ is virtually nilpotent.
\end{lem}

\begin{proof}
    We prove the contrapositive. Choose a representative for each $g \in G$ with determinant 1. Suppose that $t \in G$ has an eigenvalue $\lambda$ that is not a root of unity. Let $K$ be the subfield of $k$ generated by the prime subfield of $k$, the matrix entries of elements of a finite generating set for $G$ and the eigenvalues of $t$ in $k$. Then $G < PGL_3(K)$ and $K$ is a finitely generated field. Since $\lambda$ is not a root of unity and $K$ is finitely generated, by \cite[Theorem 2.64]{KapDrutu}, there exists an extension of $K$ to a local field $(\mathbb{K}, \vert \cdot \vert)$ such that $\vert \lambda \vert \neq 1$. By replacing $t$ with $t^{-1}$ if necessary, we can assume that $\vert \lambda \vert > 1$ and that $\lambda$ has algebraic multiplicity 1. Then $t$ is a proximal element with respect to the local field $\K$ that is an extension of the finitely generated subfield $K$ of $k$.

    Hence, if $G$ does not have a proximal element with respect to any local extension of a finitely generated subfield of $k$, then all $g \in G$ have eigenvalues roots of unity. By \cite[Theorem 14.46]{KapDrutu}, we have that $G$ is virtually nilpotent. 
\end{proof}

\begin{lem}
    \label{lem: no very proximal implies virtual fixed point on P}
    If $G$ does not have a very proximal element with respect to any local extension of a finitely generated subfield of $k$, then there exists a field $\K$ containing a finitely generated subfield of $k$ and a finite index subgroup of $G$ which fixes a point of $\mathbb{P}^2(\K)$. 
\end{lem}

\begin{proof}
    If $G$ does not have a very proximal element, then either $G$ has a proximal element or it does not. 

    If $G$ does not have a proximal element with respect to any local extension of a finitely generated subfield of $k$, then by Lemma \ref{lem: No proximal implies virtually nilpotent}, $G$ is virtually nilpotent and hence virtually solvable. Thus, by Corollary \ref{cor: virtually solvable implies elliptic}, $G$ has a finite index subgroup fixing a point of $P$. 

    Suppose that $G$ has a proximal element with respect to some local extension $\K$ of a finitely generated subfield $K$ of $k$ such that $G < PGL_3(K)$, but $G$ does not have a very proximal element with respect to $\K$. Then by \cite[Proposition 15.9]{KapDrutu} and \cite[Proposition 15.16]{KapDrutu}, the Zariski closure $\overline{G}$ is reducible or $G \curvearrowright \K^3$ is reducible. Letting $\overline{G}^{o}$ denote the irreducible component of the identity in $\overline{G}$, which is a finite index subgroup of $\overline{G}$, the finite index subgroup $G \cap \overline{G}^{o}$ of $G$ satisfies $\overline{G \cap \overline{G}^{o}} = \overline{G}^o$, which is Zariski irreducible. Hence, up to passing to the finite index subgroup $G \cap \overline{G}^{o}$ of $G$, we can assume that $\overline{G}$ is Zariski irreducible, so that $G \curvearrowright \K^3$ is reducible, meaning that there is a proper $G$-invariant linear subspace $U$ of $\K^3$. 

    If $\dim U = 1$, then $G$ fixes a point of $\mathbb{P}^2(\K)$, and we are done.  
    
    Suppose that $\dim U = 2$. We have that $G \curvearrowright \mathbb{P}^2(\K) \setminus U$. Choosing a basis $(u_1, u_2)$ of $U$ and then extending this to a basis $(u_1, u_2, u_3)$ of $\mathbb{P}^2(\K)$, we have $\mathbb{P}^2(\K) \setminus U = \{[x_1:x_2:x_3] : x_3 \neq 0\} = \{[x_1:x_2:1]:x_1,x_2 \in \K\} \cong \K^2$. Furthermore, each $g = \begin{bmatrix}
        g_{11} & g_{12} & g_{13} \\
        g_{21} & g_{22} & g_{23} \\
        0 & 0 & g_{33}
    \end{bmatrix} \in G$ acts as:

    \begin{align*}
        g[x_1:x_2:1] &= [g_{11} x_1 + g_{12}x_2 + g_{13}:g_{21} x_1 + g_{22}x_2 + g_{23}:g_{33}] \\
        &= [\frac{g_{11}}{g_{33}}x_1 + \frac{g_{12}}{g_{33}}x_2 + \frac{g_{13}}{g_{33}}:\frac{g_{21}}{g_{33}} x_1 + \frac{g_{22}}{g_{33}}x_2 + \frac{g_{23}}{g_{33}}:1]
    \end{align*}

    So $g$ acts on the vector $(x_1, x_2) \in \K^2$ as the affine transformation $$(x_1, x_2) \mapsto \begin{bmatrix}
        \frac{g_{11}}{g_{33}} & \frac{g_{12}}{g_{33}} \\
        \frac{g_{21}}{g_{33}} & \frac{g_{22}}{g_{33}}
    \end{bmatrix}(x_1, x_2) + (\frac{g_{13}}{g_{33}}, \frac{g_{23}}{g_{33}})$$

    Therefore, we can identify $\mathbb{P}^2(\K) \setminus U$ with $\K^2$ via $[x_1:x_2:1] \mapsto (x_1, x_2)$ and $G$ with its image $G < \mathrm{Aff}(\K^2) = \K^2 \rtimes GL_2(\K)$ via $[g] \mapsto \frac{g}{g_{33}}$. Note that since $G < PGL_3(K)$, we can choose the representative $g$ to be in $GL_3(K)$, and so the image of $G$ is actually in $\mathrm{Aff}(K^2) = K^2 \rtimes GL_2(K)$. The identification of $\mathbb{P}^2(\K) \setminus U$ with $\K^2$ is then equivariant with respect to the map $[g] \mapsto \frac{g}{g_{33}}$. 

    Consider the image $H$ of $G$ in $GL_2(K) < GL_2(k)$. First, suppose that $H$ is virtually solvable. Then $G$ is virtually solvable, since the kernel of the map $G \to H$ is a subgroup of unipotent matrices $U_3(K)$, which is nilpotent. By Corollary \ref{cor: virtually solvable implies elliptic}, we obtain that $G$ has a finite index subgroup fixing a point of $\mathbb{P}^2(K) < P$.

    Next, suppose that $H$ is not virtually solvable. Then by the Tits alternative, there exists a non-abelian free subgroup $F_2 \cong H' < H$.  We will show that there exists an element $g \in H'$ with eigenvalues $\lambda, \lambda^{-1}$ where $\lambda \in k$ is not a root of unity. 

    Let $S = H' \cap SL_2 (k)$. Then every $g \in S$ has set of eigenvalues of the form $\{\lambda, \lambda^{-1}\}$ for $\lambda$ not a root of unity (when $g$ is non-trivial and diagonalisable over $k$) or $\{\lambda\}$ where $\lambda$ is a square root of unity (corresponding to when $g$ is trivial or $g$ is not diagonalisable over $k$). If there are no non-trivial diagonalisable elements of $S$, i.e.\ each $g \in S$ has eigenvalues all roots of unity, then by \cite[Proposition 14.44]{KapDrutu}, we have that $S$ has a finite index subgroup which is conjugate into the group $T_2(k)$, hence $S$ is virtually solvable. Since $S \triangleleft H' \cong F_2$ and is virtually solvable, we must have that $S = 1$. However, this is a contradiction, as taking any two non-commuting elements $s,t \in H'$, we have that $g = [s,t] \in S \setminus 1$. Thus, we conclude that there exists $g \in S$ with set of eigenvalues of the form $\{\lambda, \lambda^{-1}\}$ for $\lambda \in k$ not a root of unity. Lifting $g$ to $G$, we have that $g$ has Jordan normal form $\begin{bmatrix}
        \lambda & & \\
        & \lambda^{-1} & \\
        & & 1
    \end{bmatrix}$.

    Since $\lambda$ is not a root of unity, there exists a local norm $\vert \cdot \vert$ on an extension $\mathbb{F}$ of the subfield $F$ of $k$ generated by the prime subfield of $k$, the eigenvalues of $g$ and the matrix entries of a finite generating set of $G$, such that $\vert \lambda \vert \neq 1$ (we can assume $\vert \lambda \vert > 1$, passing to $g^{-1}$ if necessary). We then have that $g$ is a very proximal element with dominant eigenvalue $\lambda$ and minimal eigenvalue $\lambda^{-1}$. This is a contradiction, since we assumed that $G$ has no very proximal elements. We conclude therefore that $H$ must be virtually solvable and thus that $G$ has a finite index subgroup fixing a point of $P$. 
\end{proof}

\begin{ques}
    Is Lemma \ref{lem: no very proximal implies virtual fixed point on P} true for higher dimensional projective spaces?
\end{ques}

We now conclude the proof of the main theorem. We consider two cases on $G$. 

\begin{enumerate}
    \item If $G$ has a very proximal element with respect to some local field $\K$ extending a finitely generated subfield $K$ of $k$ such that $G < PGL_3(K)$, then by Lemma \ref{lem: very proximal}, we have that $\Gamma < \prod_{\mathbb{P}^2(\K)}^r G_0 \rtimes G$ is elliptic on $\prod_{\mathbb{P}^2(\K)}^r X_0$. Since $G < PGL_3(K)$, we can find a fixed point of $\Gamma$ in $\prod_{\mathbb{P}^2(K)}^r X_0$ by setting the coordinates outside $\mathbb{P}^2(K)$ of a fixed point $z$ of $\Gamma$ on $\prod_{\mathbb{P}^2(\K)}^r X_0$ to $x_0$. Extending the coordinates of a fixed point $x$ of $\Gamma$ in $\prod_{\mathbb{P}^2(K)}^r X_0$ to $X = \prod_P^{r} X_0$ by putting $z_p = x_p$ for $p \in \mathbb{P}^2(K)$ and $z_p = x_0$ for $p \in P \setminus \mathbb{P}^2(K)$, we obtain a fixed point $z \in X$ for $\Gamma$. 
    \item If $G$ does not have any very proximal element as in case 1, then by Lemma \ref{lem: no very proximal implies virtual fixed point on P}, we have that $G$ has a finite index subgroup which fixes a point of $\mathbb{P}^2(\K)$ for some extension $\K$ of a finitely generated subfield $K$ of $k$. By Lemma \ref{lem: virtual fixed point on P implies elliptic}, it follows that $\Gamma$ is elliptic on $\prod_{\mathbb{P}^2(\K)}^r X_0$, and hence on $\prod_{\mathbb{P}^2(K)}^r X_0$. Extending the coordinates of a fixed point $x$ of $\Gamma$ in $\prod_{\mathbb{P}^2(K)}^r X_0$ to $X = \prod_P^{r} X_0$ by putting $z_p = x_p$ for $p \in \mathbb{P}^2(K)$ and $z_p = x_0$ for $p \in P \setminus \mathbb{P}^2(K)$, we obtain a fixed point $z \in X$ for $\Gamma$. 
\end{enumerate}

\section{Appendix: Actions with virtually nilpotent projection}
\label{sec: virtually nilpotent case}

In this section, we prove Lemma \ref{lem: virtually nilpotent projection}. The results in this section are not necessary for the paper, but we believe they are of independent interest and so we include them in the appendix.  

Through this section, we fix a based set $(X_0, x_0)$ with a group $G_0$ acting decently on $X_0$ and an arbitrary index set $P$ with a group $H$ acting on $P$. We form the restricted product $X$ and the group $G^{\oplus} \curvearrowright X$, and denote $\psi(\Gamma)$ the projection of $\Gamma$ onto the $H$ factor in the semi-direct product, as in Section \ref{res prod}.

We will prove the following, generalizing \cite[Lemma 4.7]{Algebraic}:

\begin{lem}
    \label{lem: nilpotent case}
    Let $G_0 \curvearrowright X_0$ be a decent action of the group $G_0$ on the set $X_0$, with basepoint $x_0$. Let $H \curvearrowright P$ be an arbitrary action of the group $H$ on the set $P$. Let $\Gamma < G^{\oplus}$ be finitely generated and act purely elliptically on $X$. Suppose that $G:= \psi(\Gamma) < H$ is virtually nilpotent. Then $\Gamma$ fixes a point of $X$. 
\end{lem}

We begin with the following lemma, whose proof is due to Yves Cornulier. 

\begin{lem}
    \label{lem: decent actions nilpotent}
    Let $G$ be a finitely generated nilpotent group acting transitively on an infinite set $P$. Then there exists an element $t \in G$ all of whose orbits are infinite. 
\end{lem}

\begin{proof}
    Suppose otherwise and let $G$ be a counterexample of minimal nilpotency class, so that every element $g \in G$ has a finite orbit on $P$ and $G$ has minimal nilpotency class among finitely generated nilpotent groups exhibiting such an action. By the orbit stabilizer theorem, there exists an infinite index subgroup $H$ of $G$ such that $P = G/H$. By Zorn's lemma, since $G$ is finitely generated, we have that $H$ is contained in a maximal infinite index subgroup, so we can assume that $H$ is maximal infinite index in $G$. Let $N = Z(G)$ be the center of $G$. Then $G/N$ has strictly smaller nilpotency class than $G$. If $N < H$, then $G/N$ acts on $G/H=P$ with the same properties as $G \curvearrowright P$, which contradicts the minimality of the nilpotency class of $G$. We therefore have that $N$ is not a subgroup of $H$, and so $NH$ properly contains $H$, hence $NH$ has finite index in $G$. Thus, $NH/H$ is infinite and $NH \curvearrowright NH/H$ has the same properties as $G \curvearrowright P$, so we can assume that $G = NH$. Since $N = Z(G)$, we have that $H \triangleleft NH$. Thus, the property of an element having a finite orbit on $P$ is equivalent to the property of having every orbit finite on $P$. Hence, every element of $G$ has all orbits finite on $P$. But this means that the group $G/H$ is finitely generated nilpotent and torsion, hence finite. This contradicts that $H$ has infinite index in $G$. 
\end{proof}

\begin{proof}[Proof of Lemma \ref{lem: nilpotent case}]

Throughout the proof, we will abuse notation and use the same letter to denote an element of $\Gamma$ and its image in $G$.

First, suppose that $G$ is nilpotent. Let $\{\gamma_1,\ldots,\gamma_n\}$ be a finite set of generators of $\Gamma$, writing each $\gamma_i = (g_i, h_i)$ for some $g_i \in \prod_P G_0$ and $h_i \in G$. Since $\Gamma < G^{\oplus}$, we have that for each $i = 1,\ldots,n$, there are only finitely many $p \in P$ such that $(g_i)_p \notin Stab_{G_0}(x_0)$. Letting $Q$ be the union of the $G$-orbits of such $p \in P$, for each $x \in X$ such that $x_p = x_0$ for all $p \notin Q$, we have that $(\gamma_i x)_p = x_0 = x_p$ for each $i$ and each $p \notin Q$. Thus, it suffices to show that $\Gamma$ fixes a point of $X_0^{\oplus Q}$. Since $Q$ is a union of finitely many $G$-orbits, we can assume that $P$ is a single $G$-orbit, so that $G \curvearrowright P$ is transitive. If $P$ is finite, then $\Gamma$ fixes a point of $X$ by \cite[Remark 4.2]{Algebraic}, so we can assume that $P$ is infinite. Then by Lemma \ref{lem: decent actions nilpotent}, there exists $t \in G$ all of whose orbits on $P$ are infinite. Let $z$ be an arbitrary fixed point of $t$ on $X$.

By subdividing a central series for $G$, there exists a normal series $\{1\} = G_0 < G_1 < G_2 < \ldots < G_n = G$ for $G$ where $G_{i+1} / G_i$ is cyclic for each $i = 0, \ldots, n-1$. Denote a lift of a generator of $G_{i+1} / G_i$ to $G_{i+1}$ by $g_{i+1}$, so that $G_i = \langle g_1, g_2, \ldots, g_i \rangle$ for each $i = 1,\ldots, n$. Let $\Gamma_i = \psi^{-1}(G_i)$ for each $i$. 

    For each $f \in \Gamma$, let $B_f = \{b \in P: f \text{ or $f^{-1}$ is singular over } b \text{ with respect to }z\}$ and denote $\Gamma_i^2 = \langle f^2 : f \in \Gamma_i \rangle$ for each $i$.  First, we note that for each $i$, there exists $N \in \N$ such that for any $h \in G_{i+1}$ and $g \in G$, $[g^N, h] \in G_i$. Indeed, note that $G \curvearrowright G_{i+1}/G_i = \langle g_{i+1} \rangle$ by group automorphisms via:
    
    $$g \cdot h G_i = ghg^{-1} G_i$$

    for each $g \in G$ and each $h \in G_{i+1}$ (this action is well-defined since $G_{i+1}$ and $G_i$ are both normal in $G$). Since $G_{i+1}/G_i$ is cyclic, its automorphism group is finite and hence there exists $N \in \N$ such that for any $g \in G$, $g^N$ acts trivially on $G_{i+1}/G_i$, so that $[g^N, h] \in G_i$ for each $h \in G_{i+1}$ (if $G_{i+1}/G_i$ is infinite cyclic, then we can take $N = 2$). 
    
    We will show by induction on $i$ that $\Gamma_i^2(z) = z$ for each $i$ and that $G_i B_{f^2}$ is finite for each $f \in \Gamma$. 

    For the base case $i = 0$, let $f \in \Gamma_0 = \ker \psi$. Then on $P$, we have $f = 1$. We need to show that $f$ is biregular over each $p \in P$ with respect to $z$. Let $p \in P$. Since $\langle t \rangle p$ is infinite and since $B_{f^2}$ is finite, there exists $n \in \N$ such that for all $M \neq 0$, we have $(t^nf^2)^M(p) = (t^n)^M (p) \notin B_{f^2}$. Since $t(z) = z$, arguing as in the proof of \cite[Lemma 4.7]{Algebraic}, we have that $f^2$ is biregular over $p$ with respect to $z$ (in fact, $f$ is biregular over $p$ with respect to $z$). Hence, $f^2(z) = z$. Therefore, $\Gamma_0^2 (z) = z$. 

    That $G_0 B_{f^2}$ is finite for each $f \in \Gamma$ is true because $G_0 = \{1\}$ and $B_{f^2}$ is finite for each $f \in \Gamma$. 

   Now suppose that $\Gamma_i^2(z) = z$ and that $G_iB_{f^2}$ is finite for each $f \in \Gamma$. First, we show that $\Gamma_{i+1}^2(z) = z$.  Let $f \in \Gamma_{i+1}$ and let $p \in P$. Recall that there exists $N \in \N$ such that $[t^N,f^2] \in G_i$. To simplify notation below, denote $s: =t^N$. 

   We will show that there exists $n \in \N$ such that for all $M \neq 0$, 

    $$(s^n f^2)^M (p) \notin B_{f^2}$$

    We consider two cases. 
    
    \vs
    
    First, suppose that there exist $k, l \in \Z$ with $l \neq 0$ such that $s^k (p) = (f^2)^{l}(p)$. Since $G_iB_{f^2}$ is finite and $\langle s \rangle p$ is infinite, there exists $m_0 \in \N$ such that for all $m \geq m_0$, $s^{\pm m} (p) \notin (f^2)^j G_i B_{f^2}$ for any $j$ with $\vert j \vert < \vert l\vert$. 

    Set $n \in \N$ greater than $m_0 + \vert k \vert$. Then for all $M \neq 0$, we have: 

    \begin{align*}
        (s^n f^2)^M (p) &\in s^{nM} (f^2)^{M} G_i p \text{, since $[s, f^2] \in G_i$} \\
        &= s^{nM} (f^2)^j (f^2)^{ql} G_i p \text{, writing } M = ql + j \text{ for some $q \in \Z$ and $0 \leq j < \vert l \vert$} \\
        &= s^{nM} (f^2)^j G_i s^{qk}(p) \text{, since $s^k p = (f^2)^l p$ and $G_i \triangleleft G$}\\
        &= s^{nM+qk} (f^2)^j G_i p
    \end{align*}

    Then if $(s^n f^2)^M (p) \in B_{f^2}$, from above we would obtain: 

    $$s^{nM+qk} (f^2)^j (p) \in G_i B_{f^2} \implies s^{nM+qk}(p) \in (f^2)^{-j} G_i B_{f^2}$$

    contradicting our choice of $n$, since $\vert nM + qk \vert > m_0$ for each $M \neq 0$.

    Next, suppose that there do not exist $k,l$ as in the above case. Then since $G_iB_{f^2}$ is finite and $\langle s \rangle p$ is infinite, there are only finitely many $k \in \Z$ such that $s^k (p) \in \langle f^2 \rangle G_i B_{f^2}$. Choose $n \in \N$ such that $n > \max\{\vert k \vert : s^k (p) \in \langle f^2 \rangle G_i B_{f^2}\}$. If $(s^nf^2)^M (p) \in B_{f^2}$ for $M \neq 0$, then we obtain $s^{nM}(p) \in \langle f^2 \rangle G_i B_{f^2}$, a contradiction. 

    Therefore, we obtain $n \in \N$ such that for all $M \neq 0$
    
    $$(s^n f^2)^M (p) \notin B_{f^2}$$

    Since $s(z) = z$, we obtain that $f^2$ is biregular over $p$ with respect to $z$ following a similar induction argument as in Section \ref{sec: NSD element}. Therefore, $\Gamma_{i+1}^2(z) = z$. 

    Now we show that $G_{i+1} B_{f^2}$ is finite for each $f \in \Gamma$. If $G_{i+1} B_{f^2}$ were infinite, then by the finiteness of $B_{f^2}$ and Lemma \ref{lem: decent actions nilpotent}, there exists $b \in B_{f^2}$ and $g\in G_{i+1}$ such that $\langle g \rangle b$ is infinite. In particular, note that this implies $G_{i+1}/G_i$ is infinite cyclic, since if not then $g^N \in G_i$ for some $N$ and so $\langle g^N \rangle b$ would be an infinite orbit of an element of $G_i$ in $B_{f^2}$, contradicting that $G_i B_{f^2}$ is finite. Since $G_{i+1}/G_i$ is infinite cyclic, we have that $[f^2, g^2] \in G_i$. By the above arguments, since $[f^2, g^2] \in G_i$, $G_i B_{f^2}$ is finite, $\langle g^2 \rangle b$ is infinite and since $g^2(z) = z$, we obtain that $f^2$ is biregular over $b$ with respect to $z$, contradicting that $b \in B_{f^2}$. Thus, we conclude that $G_{i+1} B_{f^2}$ is finite for each $f \in \Gamma$. 

    At the end of this induction, we obtain that $\Gamma^2 (z) = z$. Note that $\Gamma^2 \triangleleft \Gamma$ and that $\Gamma / \Gamma^2$ is a finitely generated, 2-torsion (hence, abelian) group, and therefore it is finite. Thus, the orbit $\Gamma z$ is finite, and so $\Gamma$ fixes a point of $X$ by \cite[Remark 4.2]{Algebraic}. 

    Lastly, if $G$ is virtually nilpotent, then denoting $N < G$ a nilpotent finite index subgroup of $G$ and letting $\Gamma_N = \psi^{-1}(N)$, we have that $\vert \Gamma :  \Gamma_N \vert < \infty$ and $\Gamma_N$ fixes a point $x$ of $X$ by above, so $\Gamma x$ is finite. By \cite[Remark 4.2]{Algebraic}, we have that $\Gamma$ fixes a point of $X$. 
    
\end{proof}

We present another general case of when $G^{\oplus}$ acts decently on $X$, based on the algebraic structure of $G = \psi(\Gamma)$, though we will not need the following lemma in the proof of the main theorem. 

\begin{lem}
    \label{normal nilpotent}
    Let $G_0 \curvearrowright X_0$ be a decent action of the group $G_0$ on the set $X_0$. Let $H \curvearrowright P$ be an arbitrary action of the group $H$ on the set $P$. Let $\Gamma < G^{\oplus}$ be finitely generated and act purely elliptically on $X$. Suppose that there exists an abelian normal subgroup $N \triangleleft G:= \psi(\Gamma)$ and $t \in N$ with only finitely many finite orbits on $P$. Then $\Gamma$ fixes a point of $X$. 
    
\end{lem}

\begin{proof}
    
    We may assume that $t$ has all orbits infinite on $P$. Indeed, let $Q$ be the union of all finite orbits of $t$ on $P$. Since $N$ is abelian, we have that $NQ = Q$. It then follows that $GQ = Q$, since for any $g \in G$ and any $q \in Q$, we have: 

    $$\langle t \rangle g q = g \langle g^{-1} t g \rangle q \subset gNQ = gQ$$

    and the latter set is finite. Therefore, $Q$ is a finite $G$-invariant subset of $P$, so by \cite[Remark 4.2]{Algebraic}, we have that $G$ fixes a point of $X_0^{\oplus Q}$. Thus, it suffices to find a fixed point of $G$ on $X_0^{\oplus P \setminus Q}$, and $t$ has all orbits infinite on $P \setminus Q$. Thus, we may assume that $t$ has all orbits infinite on $P$. 
    
    Let $z$ be an arbitrary fixed point of $t$ on $X$. Since $N$ is abelian and $t$ has all orbits infinite on $P$, the proof of \cite[Lemma 4.7]{Algebraic} or Lemma \ref{lem: nilpotent case} shows that $\Gamma_N := \psi^{-1}(N)$ fixes $z$. Since $N \triangleleft G$, we have $\Gamma_N \triangleleft \Gamma$ and so $\Gamma$ acts on the fixed point set of $\Gamma_N$. Since $t$ has all orbits infinite, by Lemma \ref{Inf orb fix pt}, we have that $z$ is the unique fixed point of $t$, hence of $\Gamma_N$, so that the fixed point set of $\Gamma_N$ is $\{z\}$. Thus, $\Gamma$ acts on $\{z\}$, and hence fixes $z$. 
\end{proof}

We do not know if Lemma \ref{lem: nilpotent case} generalizes to virtually solvable projections for arbitrary actions $H \curvearrowright P$, and we leave this problem for interested readers to explore. 

\begin{ques}
    \label{ques: virtually solvable projection}
    Let $G_0$ be a group acting decently on $X_0$. Let $H$ be a group acting on $P$, and let $\Gamma < G^{\oplus}$ be finitely generated and act purely elliptically on $X$. If $\psi(\Gamma)$ is virtually solvable then does $\Gamma$ fix a point of $X$?
\end{ques}

Note that for $P = \mathbb{P}^2(k)$ for an arbitrary field $k$ and $H = PGL_3(k)$, Question \ref{ques: virtually solvable projection} has an affirmative answer by Corollary \ref{cor: virtually solvable implies elliptic}.

\bibliographystyle{plain} 
\bibliography{refs2}

\end{document}